\documentclass{article}

\usepackage{amsmath,amsthm,amsopn,amsfonts,graphicx,amssymb,dsfont,color,hyperref}
\usepackage{gensymb} 
\usepackage{textcomp}
\usepackage{mathrsfs}
\usepackage{color}
\usepackage{geometry}
\usepackage{subcaption}
\usepackage{float}
\usepackage{enumitem}
\usepackage{tikz} 

\newtheorem{thm}{Theorem}[section]
\newtheorem{prop}{Proposition}[section]
\newtheorem{lemma}{Lemma}[section]

\newtheorem{ass}{Assumption}[section]

\theoremstyle{definition}

\theoremstyle{remark}
\newtheorem{rem}{Remark}[section]

\newcommand{\R}{\mathbb{R}}
\newcommand{\N}{\mathbb{N}}
\newcommand{\dt}{\partial_t}
\newcommand{\dth}{\partial_\theta}
\newcommand{\dthth}{\partial_{\theta\theta}}

\newcommand{\abs}[1]{\left\vert #1 \right\vert}
\newcommand{\norm}[1]{\left\Vert #1 \right\Vert}

\newcommand{\1}{\mathbf{1}} 

\newcommand{\rinf}{\underline{r}}

\newcommand{\rmax}{r_{\max}}
\newcommand{\ds}{\displaystyle}

\newcommand{\ub}{\overline{u}}

\newcommand{\rhod}{\underline{\rho}}

\numberwithin{equation}{section}

\newcommand{\dedicace}[1]{%
  \begin{center}
    \vspace*{1em}
    \textit{#1}
    \vspace*{1em}
  \end{center}
}

\definecolor{aquamarine}{rgb}{0.13, 0.68, 0.8} 
\definecolor{darkgreen}{rgb}{0.2, 0.5, 0.2} 

\title{Existence of {\it two} thresholds in a bistable equation with nonlocal competition}

\author{
  Matthieu Alfaro$^{\text{a}}$, 
  Cédric Chane Ki Chune$^{\text{b,c}}$, 
  Lionel Roques$^{\text{b}}$ \\[1ex]
  \footnotesize $^{\text{a}}$ Univ. Rouen Normandie, CNRS, LMRS UMR 6085, F-76000 Rouen, France, \\ \footnotesize  $\&$ Univ. Lille, CNRS, UMR 8524, LPP, F-59000 Lille, France. \\ 
  \footnotesize $^{\text{b}}$ INRAE, BioSP, 84914 Avignon, France.\\ 
  \footnotesize $^{\text{c}}$   INRAE, VetAgro Sup, UMR EPIA, 69280 Marcy l’Etoile, France.
}

\date{}

\begin{document}

\maketitle

\dedicace{To Prof. Hiroshi Matano, an inspiring mathematician and friend.}

\begin{abstract}  We consider a nonlocal bistable reaction-diffusion equation, which serves as a model for a population structured by a phenotypic trait, subject to mutation, trait-dependent fitness, and nonlocal competition. Within this replicator-mutator  framework, we further incorporate a \lq\lq pseudo-Allee effect'' so that the long time behavior (extinction vs. survival)  depends on the  size of the initial data.   

After proving the well-posedness of the associated Cauchy problem, we investigate its long-time behavior.  We first show that small initial data lead to extinction. More surprisingly, we then prove that that extinction may also occur for too large initial data, in particular when selection is not strong enough. Finally, we  exhibit situations where intermediate initial data lead to persistence, thereby revealing the existence of (at least) {\it two} thresholds.  These results stand in sharp contrast with the behavior observed in local bistable equations.
\\


\end{abstract}


\section{Introduction}\label{s:intro}

We consider $u=u(t,\theta)$ solving the nonlocal reaction-diffusion equation
\begin{equation} \label{eq:main}
    \partial_t u = \partial_{\theta \theta} u + r(\theta) \, u - u \, \rho_u(t) - f(u), \quad t > 0, \, \theta \in \mathbb{R},
\end{equation}
with
\begin{equation}
    \rho_u(t) := \int_{\mathbb{R}} u(t, \theta ) \,  d\theta ,
\end{equation}
and starting from a nonnegative initial condition $u_0$. We assume that $r\le \rmax$ for some constant $\rmax>0$. The precise assumptions on $f$ will be detailed in the next section. For now, we simply note that $f \ge 0$
 satisfies $f(0)=0$ and $f(s) > \rmax\, s$ on the interval $(0, \varepsilon)$, for some $\varepsilon > 0$, see an example in Fig.~\ref{fig:schematic_function_f}.

\begin{figure}[h!]
    \centering
    \begin{tikzpicture}[scale=2, >=latex]
        \draw[->] (-0.0,0) -- (3,0) node[below] {$s$};
        \draw[->] (0,-0.2) -- (0,1.) node[left] {$f(s)$};

        \coordinate (O) at (0,0);        
        \coordinate (Eps) at (1,0);     
        \coordinate (DeuxEps) at (2,0);   

        \draw[dashed] (0,0) -- ++(2,0.84) node[pos=0.9,above] {$\rmax\,s$};

        \draw[thick,blue]
    (0,0)
    .. controls (0.3,0.3) and (0.7,0.9) .. (1,0.9)  
    .. controls (1.3,0.9) and (1.7,0.1)  .. (2,0);
        \draw[thick,blue] (2,0) -- (2.8,0);

        \draw[dotted] (Eps) -- ++(0,1);
        
        \node[below left] at (O) {$0$};
        \node[below] at (Eps) {$\varepsilon$};
       
        \node[below] at (DeuxEps) {$2\varepsilon$};
    \end{tikzpicture}
    \caption{Schematic representation of a bump function $f$ inducing an Allee effect. }
    \label{fig:schematic_function_f}
\end{figure}
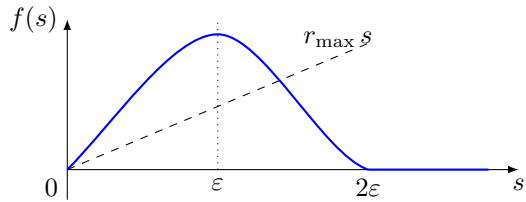

\paragraph{Biological context and motivation.} Equation \eqref{eq:main} describes the dynamics of a population density of competing individuals. The variable $\theta$ represents  some  phenotypic trait of individuals. The diffusion term $\partial_{\theta \theta}$ represents the effect of mutations. The fitness of individuals with phenotype $\theta$ is described by the function $r(\theta)$. Competition among all individuals leads to the nonlocal term $\rho_u(t)$. The function $f$ induces negative fitness at low population density, modeling a strong Allee effect due to the assumption $f(s) > \rmax s$ for sufficiently small $s$. In the example of Fig.~\ref{fig:schematic_function_f}, we assume that $f$ has no further effect when $u(t,\theta) \geq 2\varepsilon$.

We argue that the standard replicator-mutator equation (see, e.g., \cite{AlfCar17, MarRoq16, TsiLev96}), which describes the dynamics of phenotypically structured populations but with $f \equiv 0$ (see below for a more detailed description), does not fully capture certain essential aspects of adaptive invasion into a new environment, i.e., invasion that requires the population to evolve and adapt to the new environment. Here, the term ``environment" is understood broadly:   among others, an application  we have in mind concerns the emergence of zoonotic diseases  which represent, according to WHO,  around 60$\%$ of the human diseases. 
Specifically, we aim to describe conditions under which an introduced population $u_0$ (e.g. of pathogens), originating from a reservoir host and introduced into a new host, successfully adapts, causing  {\it spillover}. As for influenza A, we  refer, among others, to  \cite{Par-et-al-15, Web-Web-01}, while the works \cite{Lat-et-al-20, Lau-et-al-05} are concerned with 2003 and 2019 coronavirus outbreaks.

In the standard replicator-mutator model $
\partial_t u = \partial_{\theta \theta} u + r(\theta) \, u - u \, \rho_u(t)$, 
 the persistence of the introduced population is independent of $u_0$—in particular, it does not depend on the initial population size or its phenotypic characteristics. Indeed, survival of the introduced population depends exclusively on the sign of the principal eigenvalue of the operator $\phi \mapsto \phi'' +   r(\theta)\phi$ \cite{AlfVer18}. This phenomenon of independence from initial conditions arises due to the well-known mathematical property, often considered biologically unrealistic, of infinite speed of propagation of the solution's support. Mathematically, this property follows directly from the strong maximum principle \cite{friedman-parabolic}, which instantaneously generates mutants possessing the entire spectrum of possible phenotypes—including those well-adapted to the new host (i.e., phenotypes close to the  fitness  optimum). Combined with the ``atto-fox problem" \cite{Mol91}, this effect allows infinitesimally small subpopulations (with densities $\ll 1$) possessing advantageous phenotypes to grow whenever the fitness function $r$ on the new host permits.

In this work, we introduce the bump function $f$ precisely to prevent growth at very low population densities. This mechanism acts as a ``pseudo-Allee effect": the function $f$ does not represent a biologically driven Allee effect (such as cooperation among individuals) but rather serves to counteract the unrealistic instantaneous appearance of well-adapted phenotypes in extremely small quantities.

We will demonstrate how this equation, equipped with the additional term $f$, differs in its dependence on initial conditions from the classical replicator-mutator model. Moreover, we will illustrate how its properties also differ significantly from classical bistable equations, which do not include nonlocal interactions.

\paragraph{The case $f\equiv 0$: replicator-mutator equation.}
In the absence of the Allee effect, the ``replicator-mutator" equation,  which originates from the works \cite{Fle-79, Kim-65, Lan-75}, 
\begin{equation} \label{eq:replicator-mutator}
    \partial_t u = \partial_{\theta \theta} u + r(\theta) \, u - u \, \rho_u(t) \quad t > 0, \, \theta \in \mathbb{R},
\end{equation}
has been widely studied in recent years, particularly when the fitness term $r(\theta)$ is quadratic \cite{ AlfCar17, AlfVer18, HamLav20,MarRoq16}, $r(\theta) = \rmax - \alpha ^2 \theta^2$ ($\alpha >0$ measuring the strength of selection),
which corresponds to the assumptions of Fisher’s geometric model. These works have yielded analytical solutions for $u(t,\theta)$, especially simple when the initial datum is Gaussian,  see also \cite{Bik-14}. The stationary states $p$ of the equation satisfy
$ \partial_{\theta \theta} p +   r(\theta) p= \lambda\, p$, with $\lambda=\rho_p =\int_\R p(\theta) d\theta $.
As noted in \cite{AlfVer18}, these stationary states are necessarily eigenfunctions of the Schrödinger operator $\phi \ \mapsto \ \partial_{\theta \theta}  \phi + r(\theta)\phi$, with $\lambda$ as the associated eigenvalue. Since the potential $r(\theta)$ is confining due to its quadratic decay, we are assured of the existence and uniqueness of the principal eigenvalue $\lambda$ and the principal eigenfunction $p > 0$, with the uniqueness of the latter understood up to a multiplicative constant. With $r(\theta) = \rmax - \alpha^2 \theta^2$, the principal eigenfunction is Gaussian, centered at the optimum $0$:
\begin{equation} \label{eq:gaussian_p}
	\ds p(\theta) = C\, \frac{\alpha^{1/2}}{\sqrt{2 \, \pi}} e^{-\alpha \frac{\theta^2}{2}},
\end{equation}
with $C>0$ an arbitrary constant, and 
\begin{equation} \label{eq:eigenvalue_rho}
	\lambda = \rmax -  \alpha.
\end{equation}
For this pair $(\lambda, p)$ to be a valid stationary solution of \eqref{eq:replicator-mutator}, we must ensure that condition $\lambda=\rho_p$ is satisfied. For this, it is necessary that $\lambda > 0$, i.e., $\rmax > \alpha $, and then we must take $C = \lambda$.

Let us focus on the maximum value taken by the equilibrium solution:
\begin{equation} \label{eq:pinfty}
	\|p\|_\infty = (\rmax -  \alpha ) \frac{\alpha^{1/2}}{\sqrt{2 \, \pi}}.
\end{equation}
We note that a very low selection $\alpha \to 0$ leads to $\|p\|_\infty\to 0$, as well as a high selection $\alpha \to \rmax$ leads to $\|p\|_\infty\to 0$ (if $\alpha \ge \rmax$, the stationary solution vanishes). In fact, the expression \eqref{eq:pinfty} shows that $\|p\|_\infty$ reaches its maximum when 
$\alpha=\alpha_{max}:=\rmax/3.$
We thus observe a non-monotonic dependence of $\|p\|_\infty$ with respect to $\alpha$, and the existence of an optimal positive selection coefficient. This suggests that in the presence of an Allee effect, i.e., with the additional term $f$ in the equation \eqref{eq:main}, an increase in the selection coefficient could, in certain situations, lead to the solution crossing the critical threshold $\varepsilon$ (or $2 \, \varepsilon$) and to the persistence of a solution that would have converged to $0$ with a lower selection coefficient. 

\paragraph{The standard bistable problem \cite{DuMat10,Zla06}.} Let us now establish a connection between \eqref{eq:main} and the most classical bistable problems, with a local competition term and a constant coefficient $r$:
\begin{equation} \label{eq:bistable}
    \partial_t u = \partial_{\theta \theta} u + r\, u - u^2 - f(u), \quad t > 0, \, \theta \in \mathbb{R}.
\end{equation}
To this end, we make, solely for this digression, the assumption that  $k(s):=r\, s - s^2 -f(s)$  vanishes at $0$, at $s_1\in(\varepsilon,r)$, at $r$, is negative on $(0,s_1)$, positive on $(s_1,r)$, negative on $(r,+\infty)$  and has positive integral over $(0,r)$. It is easily checked that these assumptions can be achieved with well-chosen bump functions $f$ as in Fig.~\ref{fig:schematic_function_f}.
Then, the behavior of the Cauchy problem associated with equation~\eqref{eq:bistable} is well-known~\cite{DuMat10, Pol11, Mur-Zho-13}. Recall a few results from the literature: for any compactly supported initial data $u_0: \mathbb{R} \to [0,r]$, the solution of \eqref{eq:bistable} starting from $u_0$ can exhibit only three possible long-term behaviors as $t \to +\infty$:
\begin{itemize}
    \item Extinction: $u(t, \cdot) \to 0$ uniformly in $\mathbb{R}$.
    \item Invasion: $u(t, \cdot) \to r$ locally uniformly in $\mathbb{R}$.
    \item Convergence to a ground state: $u(t, \cdot) \to p(\cdot + x_0)$ uniformly in $\mathbb{R}$ for some $x_0 \in \mathbb{R}$, where $p: \mathbb{R} \to (0,r)$ is the unique positive stationary solution of \eqref{eq:bistable} converging to $0$ at infinity and such that $\max_{\mathbb{R}} p = p(0)$.
\end{itemize}
In order to describe threshold phenomena between extinction and invasion, one considers, as in \cite{DuMat10,Pol11}, a family $(u^\sigma_0)_{\sigma \ge 0}$ of  initial conditions in $L^\infty(\mathbb{R})$ which satisfies:
\begin{equation} 
\left\{
\begin{aligned}
    &\text{1. For each } \sigma \ge 0, \, u^\sigma_0 \text{ is compactly supported;} \\
    &\text{2. } \exists M>0 \hbox{ such that for all }\sigma\ge 0, \ \|u^\sigma_0\|_{L^\infty}\le M; \\
    &\text{3. } \sigma \mapsto u^\sigma_0 \text{ is continuous from } \mathbb{R} \text{ to } L^1(\mathbb{R}); \\
    &\text{4. } u^0_0 = 0, \ u^\sigma_0 \ge 0 \text{ and the family is increasing: if } \sigma_1 < \sigma_2, \text{ then } u^{\sigma_1}_0 \le u^{\sigma_2}_0 \text{ and } u^{\sigma_1}_0 \not\equiv u^{\sigma_2}_0 \text{ a.e.}; \\
    &\text{5. } \lim_{\sigma \to +\infty}\|u^\sigma_0\|_{L^1}= + \infty.
\end{aligned}
\right.
\label{eq:family_initial_conditions}
\end{equation}
Then, the results in \cite{DuMat10} (see also \cite{Zla06} for initial conditions which are indicator functions of intervals) show that
there is a unique threshold $\sigma^* \in (0, +\infty]$ such that, for the solution of the Cauchy problem~\eqref{eq:bistable} starting from $u^\sigma_0$, 
extinction occurs if $0 \le \sigma < \sigma^*$, invasion occurs if $\sigma > \sigma^*$, and convergence to a ground state occurs if $\sigma = \sigma^*$.

These so-called {\it sharp threshold} phenomena have  received a lot of attention. Among others, we may mention the works  
\cite{Pol11} considering possibly nonautonomous reaction terms,   \cite{Mat-Pol-16} allowing initial data with {\it tails}, \cite{Mur-Zho-13,Mur-Zho-17} relying on energy estimates in a $L^2$ framework, \cite{Alf-Duc-Fay-20}  providing a first quantitative estimate of the threshold value. 
These  results are typically  obtained for monotone
families of initial conditions, but the threshold phenomenon can also be related to the issue
of fragmentation of the initial data. With that respect, we may refer to  \cite{AlfHamRoq24, Gar-Roq-Ham-12, Nad-preprint-23}.

\paragraph{New phenomena.} For \eqref{eq:main}, the presence of the nonlocal competition term results in dynamics that are fundamentally different from the standard bistable case. Notably, due to the integral term  (and, possibly, non constant fitness), there is no constant positive stationary solution, which makes invasion impossible. The ``threshold solutions", observed in the local  case when $\sigma = \sigma^*$ \cite{DuMat10, Pol11}, appear to play a crucial role in this context.

 Furthermore,  our analysis and numerical results indicate the possible existence of two thresholds, $\sigma_*<\sigma^*$, with convergence to ground states occurring when $\sigma$ lies within the interval $[\sigma_*, \sigma^*]$, and extinction taking place when $\sigma$ is outside this interval.

\paragraph{Organization of  the paper.} In Section~\ref{s:results}, we state the precise assumptions and present the main results. Section~\ref{s:numerical} is devoted to numerical simulations, which illustrate the richness of the possible outcomes, in particular the potential existence of two sharp thresholds. The well-posedness of the Cauchy problem is established in Section~\ref{s:cauchy}. Finally, Section~\ref{s:pers-vs-ext} is dedicated to the proofs of the various results concerning extinction and persistence.

\section{Assumptions and main results}\label{s:results}

Throughout the paper, we make the following assumptions on the  function $r$  describing the fitness landscape of the pathogen in terms of its phenotype $\theta \in \R$.

\begin{ass}[Fitness function]\label{ass:r}
  We suppose $r\in C^1(\R)$ is bounded from above, with $\rmax := \sup_{\theta\in\R} r(\theta)>0$, and decays at most quadratically, i.e. there exists $C>0$ such that $r(\theta)\geq -C(1+\theta^2)$. We also suppose that, for every $\mu>0$, the function $\theta\mapsto|r'(\theta)|e^{-\mu\theta^2}$ is bounded.
\end{ass}
As for the function $f$, we assume the following.
\begin{ass}[Allee effect]\label{ass:f}
We suppose $f\in  C^1(\R)$,  $f(0)=0$, $f\ge 0$ and  $f'(0) > \rmax$. We also assume that $f$ and $f'$ are bounded, and that $f''(0)$ exists. 
\end{ass}
These assumptions on $f$ imply that
\begin{equation}\label{eq:def_epsilon}
    \exists \ \varepsilon>0 \hbox{ such that }\rmax s - f(s)<0 \hbox{ for all }s\in (0,\varepsilon],
\end{equation}
and that $f$ is Lipschitz continuous. We will note $C_{Lip}$ its Lipschitz constant.

Our first result concerns the well-posedness of the Cauchy problem associated with \eqref{eq:main}.

\begin{ass}[Initial datum]\label{ass:u0}
We suppose $u_0 \in C^{0,\beta}(\R)$  for some $\beta\in(0,1)$, and $u_0 \geq 0$. We also suppose  $u_0(\theta) \leq C_0e^{-\mu_0\theta^2}$ for some $C_0,\mu_0>0$.
 \end{ass}

\begin{thm}[Well-posedness]\label{th:well-pos}
Under Assumptions~\ref{ass:r}, \ref{ass:f}, and~\ref{ass:u0}, 
there exists a unique global solution 
\[
u \in C^0([0,+\infty)\times\R)\cap C^{1;2}_{t;\theta}((0,+\infty)\times\R)
\]
to \eqref{eq:main}  starting from $u_0$.  
Moreover, this solution satisfies:
\begin{enumerate}[label=(\roman*)]
    \item $u(t,\theta) \ge 0$ for all $(t,\theta)\in[0,+\infty)\times\R$, and if $u_0\not\equiv 0$, then $u(t,\theta)>0$ for all $t>0$ and $\theta\in\R$;
    \item for every $T>0$, there exist constants $C_T,\mu_T>0$ such that
    \[
    0 \le u(t,\theta) \le C_T e^{-\mu_T\theta^2}, \qquad (t,\theta)\in[0,T]\times\R;
    \]
    \item the quantity $\ds\rho_u(t)=\int_\R u(t,\theta)\,d\theta$ is globally bounded in $(0,+\infty)$.
\end{enumerate}
\end{thm}

We now turn to the long time behavior of the solution $u^\sigma=u^\sigma(t,\theta)$ to \eqref{eq:main} starting from $u_0^\sigma$, the family  of initial conditions $(u_0^\sigma)_{\sigma\geq 0}$ satisfying \eqref{eq:family_initial_conditions}.  To match with Assumption \ref{ass:u0} we also assume that all the $u_0^\sigma$'s belong to $C^{0,\beta}(\R)$ for some $\beta\in(0,1)$. 

First, we have the following (expected) result which asserts extinction for \lq\lq small'' initial data.

\begin{prop}[Extinction for small data]\label{prop:sigma-pt} Under Assumptions~\ref{ass:r} and~\ref{ass:f}, there is $\sigma_*>0$ such that, for any $0<\sigma <\sigma_*$, $u^\sigma \to 0$ as $t \to +\infty$ uniformly in $\mathbb{R}$.
    \end{prop}
    
    More striking is the fact that extinction may also occur for \lq\lq large'' initial data, which is in sharp contrast with the standard bistable problem, see Section \ref{s:intro}.

\begin{thm}[Extinction for large data]\label{th:sigma-gd} Assume further that $r$ is bounded from below. Then there is $\sigma^*\geq \sigma_*$ such that, for any $\sigma>\sigma^*$, $u^\sigma \to 0$ as $t \to +\infty$ uniformly in $\mathbb{R}$.
\end{thm}
    
The persistence or not of the above result when $r$ is not assumed bounded from below seems subtle. For instance, for quadratic fitness functions $r(\theta)=\rmax-\alpha^2\theta^2$, numerical simulations suggest  that this should depend on the strength of selection $\alpha>0$, see Section \ref{s:numerical}. A rigorous proof of such a numerical conjecture is quite challenging. 

Obviously, the thresholds $\sigma_*$ and $\sigma^*$ in Proposition \ref{prop:sigma-pt} and  Theorem \ref{th:sigma-gd} are meaningful only if the survival of the population is possible,  under certain conditions, for values of $\sigma$ between $\sigma_*$ and $\sigma^*$. According to the numerical simulations presented in Section \ref{s:numerical}, the following four different scenarios may occur:
\begin{enumerate}
    \item [\textbf{(E)}] Extinction in all cases: For all $\sigma > 0$, $u^{\sigma} \to 0$ as $t \to +\infty$ uniformly in $\mathbb{R}$.
    \item [\textbf{(E$^*$)}] Extinction in all cases except at a threshold value: there exists $\sigma_*>0$ such that for all $\sigma \neq \sigma_*$, $u^{\sigma} \to 0$ as $t \to +\infty$ uniformly in $\mathbb{R}$. However, $u^{\sigma_*} \to p$ as $t \to +\infty$ uniformly in $\mathbb{R}$, where $p$ is a positive stationary state of~\eqref{eq:main}.
    \item [\textbf{(EPE)}] Persistence between two threshold values: there are $0<\sigma_* < \sigma^* < +\infty$ such that for all $\sigma < \sigma_*$ and all $\sigma > \sigma^*$, $u^{\sigma} \to 0$ as $t \to +\infty$ uniformly in $\mathbb{R}$. However, for all $\sigma_*\leq \sigma \leq \sigma^*$, $u^{\sigma} \to p^\sigma$ as $t \to +\infty$ uniformly in $\mathbb{R}$, where $p^\sigma$ is a positive stationary state of~\eqref{eq:main}.
    \item [\textbf{(EP)}]  
    Persistence after a threshold value: there is $\sigma_*>0$ such that for all $0<\sigma<\sigma_*$, $u^{\sigma} \to 0$ as $t \to +\infty$ uniformly in $\mathbb{R}$. However, for all $\sigma \geq \sigma_*$, $u^{\sigma} \to p^\sigma$ as $t \to +\infty$ uniformly in $\mathbb{R}$, where $p^\sigma$ is a positive stationary state of~\eqref{eq:main}.
\end{enumerate}
 From Theorem  \ref{th:sigma-gd}, the scenario \textbf{(EP)} is excluded by fitness functions that are bounded from below. However, it may occur for quadratic fitness functions that are very narrow in the sense that the strength of selection $\alpha$ is \lq\lq large''.

Thus, contrary to the standard local bistable case, where the ground states (non-constant stationary solutions) played a minor role (as they were only achieved for a single value of $\sigma$), here the ground states seem to play a major role, see scenario \textbf{(EPE)}. Demonstrating that $\sigma_*$ and $\sigma^*$ are two sharp thresholds, with behaviors of type \textbf{(E)}, \textbf{(E$^*$)}, \textbf{(EPE)}, appears to be out of reach for now. The problem is far more challenging than the classical problem \eqref{eq:bistable} due to the nonlocal term, which makes comparison arguments difficult. As revealed  by the proof  of Theorem~\ref{th:sigma-gd}, when $\rho_u$ increases, it tends to decrease $u$, but a decrease in $u$ causes a reduction in $\rho_u$. These types of interactions not only introduce technical difficulties but are also responsible for the existence of the second threshold $\sigma^*$.

\medskip

Despite such difficulties, we are able to exhibit situations where survival does occur. To do so, we need {\it ad hoc} families of Allee effect functions $f_\varepsilon$ and fitness functions $r_\varepsilon$ that are scaled properly via  the small parameter $0<\varepsilon\ll 1$.

Recall that a  function $h:\R \rightarrow \R$ is said to be radially nonincreasing if it is even and $\theta\mapsto h(\theta)$ is nonincreasing on $[0,+\infty)$.

 \begin{ass}[{\it Ad hoc} case scaled by the strength of selection]\label{ass:perturb}  For each $\varepsilon>0$,  
  we assume that $f_\varepsilon$ satisfies Assumption~\ref{ass:f}  and that $r_\varepsilon$ satisfies Assumption~\ref{ass:r}. Moreover, we assume that $\|f_\varepsilon'\|_{L^\infty}=O(1)$  and $\|f_\varepsilon\|_{L^\infty}= O(\varepsilon)$ as $\varepsilon \to 0$. 
 As for the fitness function, we require $r_\varepsilon$ to be a radially nonincreasing fitness function such that
\begin{equation}\label{r-takes-the-form}
r_\varepsilon(\theta)\begin{cases}=\rmax -\varepsilon^{2\gamma} \theta ^2, \quad &\forall \theta \in \left(-\frac{\sqrt{2\rmax }}{\varepsilon^\gamma},\frac{\sqrt{2\rmax }}{\varepsilon^\gamma}\right),\\
\geq \rmax -\varepsilon^{2\gamma} \theta ^2, &\forall \theta \notin \left(-\frac{\sqrt{2\rmax }}{\varepsilon^\gamma},\frac{\sqrt{2\rmax }}{\varepsilon^\gamma}\right),
\end{cases}
\end{equation}
for some $0<\gamma<1$. 
 \end{ass}
 
  \begin{rem}\label{rem:gamma=1} As will become transparent in the very end of the proof of Theorem \ref{th:survival}, in \eqref{r-takes-the-form} we may replace $\varepsilon ^{2\gamma}$ ($0<\gamma<1$) by $A^2\varepsilon ^2$ for $A>0$. Then, if $A>0$ is chosen sufficiently large, the result below (for $0<\varepsilon\ll 1$) remains valid. 
 \end{rem}

\begin{thm}[Survival may occur]\label{th:survival}  Let us consider a radially nonincreasing initial condition $u_0\in C^0_c(\R, [0,+\infty))$ with $0<\rho_{u_0}<\rmax$. Then there is $\varepsilon_0>0$ such that, for any $0<\varepsilon<\varepsilon_0$, the solution $u$ to \eqref{eq:main} --- with $f=f_\varepsilon$ and $r=r_\varepsilon$ as in Assumption \ref{ass:perturb}---  starting from $u_0$ persists, in the sense that there is $\rho^*>0$ such that, for all $t\geq 0$, $\rho_u(t)\geq \rho^*$.
\end{thm}

 Obviously, \eqref{r-takes-the-form} allows quadratic  but also bounded from below fitness functions. Therefore it follows from Proposition \ref{prop:sigma-pt}, Theorem \ref{th:sigma-gd} and Theorem \ref{th:survival} that there are situations where the  two thresholds  do exist. This sustains the possibility of scenarios \textbf{(E$^*$)} and \textbf{(EPE)}.

We now present a result for quadratic fitness functions which proves the existence of a range of selection intensity within which solutions may persist, no matter how large the initial datum is.
This contrasts with the case where the fitness is bounded from below, since initially large solutions will always go extinct, as shown by Theorem \ref{th:sigma-gd}.
This highlights the non-monotonic nature of the model: although a quadratic fitness function may take smaller values than one that is bounded from below, it can nevertheless rescue a population otherwise doomed to extinction.  Again, we emphasize that the reason for this is the subtle interplay between $u$ and $\rho_u$ mentioned above.

\begin{ass}[{\it Ad hoc} case scaled by the maximal fitness] \label{ass:large select}
    Suppose the fitness function is of the form $r(\theta) = \rmax - \alpha^2\theta^2$, with 
\begin{equation}
    1\leq r_{max}\leq \alpha^2 \leq 2 r_{max},
\end{equation}    
    and the Allee effect function satisfies Assumption \ref{ass:f} together with $\norm{f}_{L^\infty([0,+\infty))} \leq 2\rmax$, and $\norm{f'}_{L^\infty([0,+\infty))} \leq 2\rmax$.
    Suppose we also have a family $(u_0^\sigma)_{\sigma\geq 0}$ of radially nonincreasing initial data satisfying the conditions (\ref{eq:family_initial_conditions}), and such that, for all $\sigma\geq 1$, $\norm{u_0^\sigma}_{L^\infty(\R)} \leq 2\rmax$ and $u_0^\sigma(\theta) \geq \rmax$ for all $\theta\in[-\frac{1}{2}, \frac{1}{2}]$.
\end{ass}

Obviously the above assumption is scaled by $r_{max}$. We note that, for any $\rmax>0$, it is always possible to choose $f$ and $(u_0^\sigma)_{\sigma\geq0}$ satisfying the above hypotheses. For instance, we may set $f(u) = \rmax u e^{\frac{1}{2}-u}$ for all $u\geq0$, and $u_0^\sigma(\theta) = \frac{3}{2}\rmax\varphi(\frac{\theta}{\sigma})$ for all  $\sigma\geq 1$ and $\theta\in\R$, where $\varphi$ is a radially nonincreasing smooth function such that $\varphi_{|[-1,1]} \equiv 1$ and $\varphi_{|\R\setminus(-2,2)} \equiv 0$.

\begin{thm} [Survival for large selection] \label{th:survival large alpha}
    Let Assumption \ref{ass:large select} be satisfied. Then, if $\rmax$ is large enough, the corresponding solution $u^\sigma$ persists for all $\sigma\geq 1$.
\end{thm}

Our last evidence of the possible existence of two thresholds consist in proving, under some conditions, the co-existence of two ordered stationary solutions when the fitness function is constant. Here, we use a more specific form of the function $f$, assigning it a bump-like shape as in Fig.~\ref{fig:schematic_function_f}:
\begin{ass}[Bump function]\label{ass:bump} In addition to Assumption~\ref{ass:f}, there is $\varepsilon>0$ such that 
\begin{equation*} 
\left\{
\begin{aligned}
    &f \equiv 0   \text{ in }  (2\varepsilon, \infty),\\
    & f' \ge 0 \hbox{ in } (0,\varepsilon) \text{ and }   f' \le 0   \text{ in } (\varepsilon, 2\varepsilon), \\ & \rmax\, s - f(s) <0 \hbox{ for all } s \text{ in } (0, \varepsilon).
\end{aligned}
\right.
\end{equation*}
\end{ass}

\begin{thm}[Existence of two stationary states]\label{th:etats-stats}  Assume  that $r$ is constant, $r(\theta)=r= \rmax$. Assume that $f$ satisfies Assumption~\ref{ass:bump}. If 
\begin{equation}\label{eq:hyp_tech_f}
  2\varepsilon^2 \le \frac{1}{r}\int_0^{2\varepsilon} f <2 \varepsilon^2 + \frac{r^3}{128}   -\frac{\sqrt{2} \varepsilon}{8} r^{3/2},  
\end{equation}
equation~\eqref{eq:main} admits two bounded, positive and radially decreasing stationary solutions, $p_1,$ $p_2$ with $p_1<p_2$. 
\end{thm}
 
Assume for instance that $f$ is a triangular bump function of the form:
\[
f(s) = 
\begin{cases}
2r \left( \varepsilon - |s - \varepsilon| \right) & \text{if } |s - \varepsilon| < \varepsilon, \\ 
0 & \text{otherwise}.
\end{cases}
\]
This function can be smoothed so that it fulfills Assumptions~\ref{ass:f} and~\ref{ass:bump}. Moreover,
\[
\frac{1}{r} \int_0^{2\varepsilon} f(s)\, ds = 2\varepsilon^2,
\]
so the left inequality in~\eqref{eq:hyp_tech_f} is always satisfied, and the right inequality is satisfied as soon as $r > \left(16 \sqrt{2} \, \varepsilon\right)^{2/3}$.

\section{Numerical explorations}\label{s:numerical}

In this section, we investigate persistence vs.\ extinction for the Cauchy problem associated with~\eqref{eq:main} using a simple two-parameter family of compactly supported initial data.
For $L\ge 0$ (the support length) and $H>0$ (the initial height), we set
\[
u_0^{L,H}(\theta)\;=\;H\,\1_{(-L/2,\,L/2)}(\theta).
\]
For each fixed $H>0$, the family $(u_0^{L,H})_{L\ge 0}$, indexed by $\sigma=L$, satisfies the hypotheses in~\eqref{eq:family_initial_conditions}
(up to the H\"older regularity assumed in Assumption~\ref{ass:u0}).  The precise assumptions on $f$ and $r$ are described in the legend of Fig.~\ref{fig:extinction}.

\paragraph{Discretization and time stepping.}
All computations were carried out with the Readi2Solve toolbox~\cite{readi2solve}.
Equation~\eqref{eq:main} is discretized in space by second-order finite differences on a uniform grid in a finite domain $(-40,40)$ with homogeneous Dirichlet boundary conditions, yielding a semi-discrete method-of-lines ODE system.
Time integration is performed with \texttt{odeint} from \textsc{SciPy}, which uses \textsc{LSODA} to adaptively switch between Adams and \textsc{BDF} schemes.
A reproducible Python Jupyter notebook is available  \href{https://doi.org/10.17605/OSF.IO/VNZRH}{here}.

\paragraph{Diagnostic for persistence vs.\ extinction.}
In all experiments we evaluate $\max_{\theta} u(T,\theta)$ at the final time $T=5$.
If this maximum is below the Allee threshold, $\max_{\theta} u(T,\theta)<\varepsilon$, we classify the trajectory as heading toward extinction.
If it exceeds the upper plateau, $\max_{\theta} u(T,\theta)>2\varepsilon$, we classify it as persisting at later times.
We therefore plot, in Fig.~\ref{fig:extinction}, the map $(H,L)\mapsto \max_{\theta} u(T,\theta)$, which reveals two sharply contrasted regions corresponding to persistence  (in yellow)   and extinction  (in purple).
Along the narrow boundary separating these two regions, the long-term behavior remains uncertain and would require simulations with larger $T$ to be clarified; however, since this boundary zone is very thin, such refinements are not essential for our purposes.

\paragraph{Constant fitness.}
We first consider a constant fitness $r(\theta)\equiv r=\rmax$, as in Theorem~\ref{th:etats-stats}; results are shown in Fig.~\ref{fig:extinction}a.
For small $H$ we observe scenario \textbf{(E)} (extinction in all cases).
As $H$ increases there appears to be a critical height $H^*$ at which scenario \textbf{(E$^*$)} emerges (extinction in all cases except at a threshold).
For $H>H^*$ we then observe scenario \textbf{(EPE)} (persistence for an interval of $L$ delimited by two thresholds).

\paragraph{Fitness bounded from below.}
Next, we consider a fitness function $r$ with a maximum at $\theta=0$, which becomes negative away from the optimum while remaining bounded from below, in accordance with the assumptions of Theorem~\ref{th:sigma-gd}; see Fig.~\ref{fig:extinction}b.
Qualitatively, we again observe a transition \textbf{(E)} $\to$ \textbf{(E$^*$)} $\to$ \textbf{(EPE)} as the height $H$ increases, mirroring the constant-fitness case.

\paragraph{Quadratic fitness.}
Finally, we consider the quadratic landscape $r(\theta)=\rmax-\alpha  ^2 \theta^2$, which does not satisfy the ``bounded from below'' hypothesis of Theorem~\ref{th:sigma-gd}.
For weak selection (small $\alpha$), the outcome is qualitatively similar to the previous two settings: as $H$ increases we again observe the sequence \textbf{(E)} $\to$ \textbf{(E$^*$)} $\to$ \textbf{(EPE)}; see Fig.~\ref{fig:extinction}c.  

For stronger selection  (larger $\alpha$), the extinction-for-large-$L$ phenomenon observed in the other cases may fail.
In this regime, increasing $H$ leads to the extended progression \textbf{(E)} $\to$ \textbf{(E$^*$)} $\to$ \textbf{(EPE)} $\to$ \textbf{(EP)}; see Fig.~\ref{fig:extinction}d.
This demonstrates that the assumption of $r$ being bounded from below is not merely technical in Theorem~\ref{th:sigma-gd}: when it is not satisfied, the conclusion of the theorem may indeed fail.

\medskip
Overall, these computations support the qualitative picture developed in Sections~\ref{s:intro} and~\ref{s:results}: depending on the fitness landscape and on the initial \lq\lq mass--support trade-off'' $(H,L)$, the dynamics can exhibit one or two thresholds separating extinction from persistence, with the nonlocal competition term playing a key role in producing the ``two-thresholds'' \textbf{(EPE)} regime. A natural question is which stationary state is reached in the persistent regime as $L$ varies.
Our numerical simulations suggest that solutions consistently converge to the same stationary state (Fig.~\ref{fig:final_profiles}).
By contrast, Theorem~\ref{th:etats-stats} guarantees, in the constant $r$ case, the existence of (at least) two distinct stationary states. The numerical evidence therefore points to only one of them being stable.

\begin{figure}[htbp]
  \centering

  \begin{subfigure}[t]{0.48\textwidth}
    \centering
    \includegraphics[width=\linewidth]{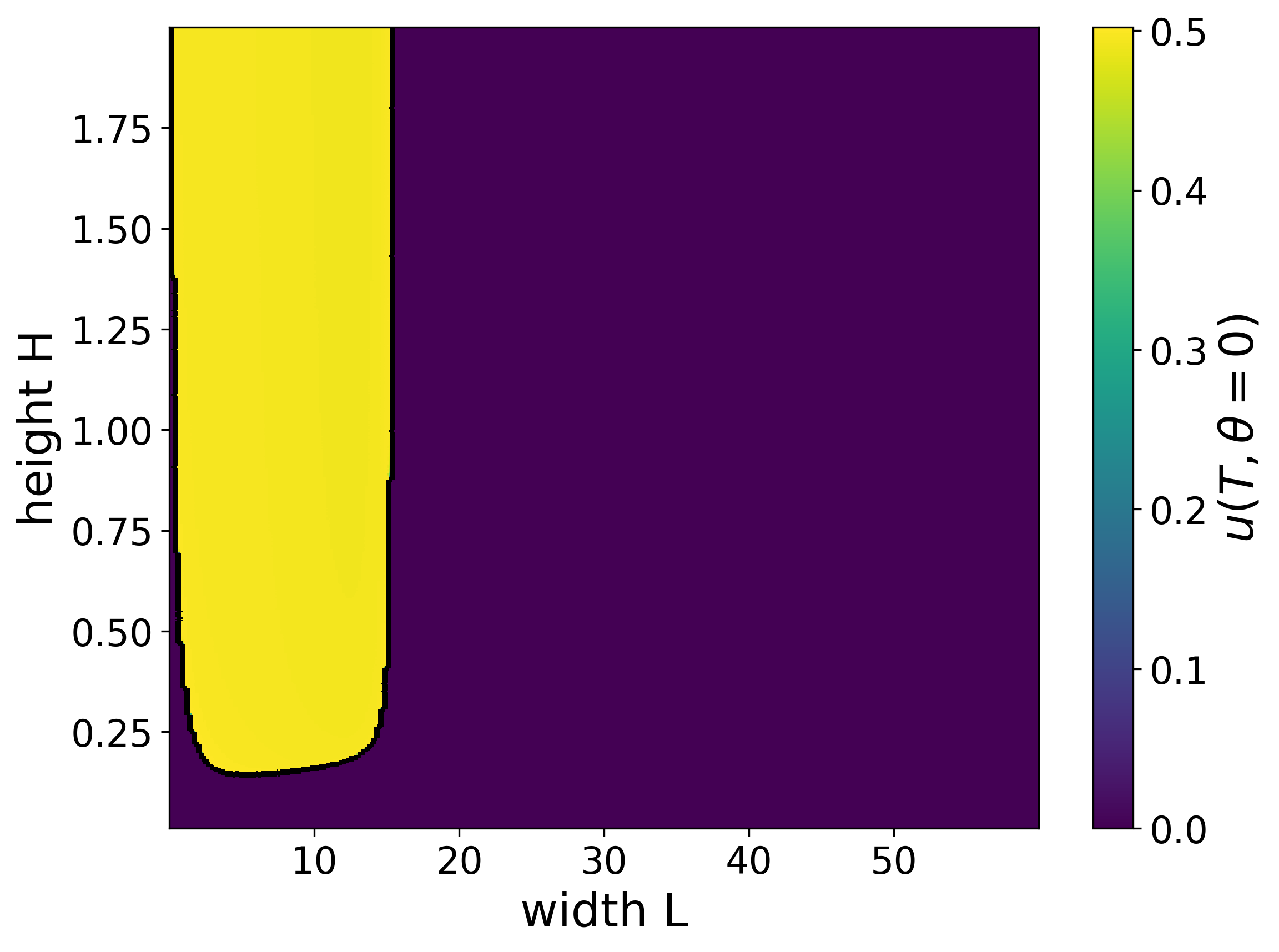}
    \subcaption{$r(\theta) = 2$}
    \label{fig:ext4-constant}
  \end{subfigure}
  \hfill
  \begin{subfigure}[t]{0.48\textwidth}
    \centering
    \includegraphics[width=\linewidth]{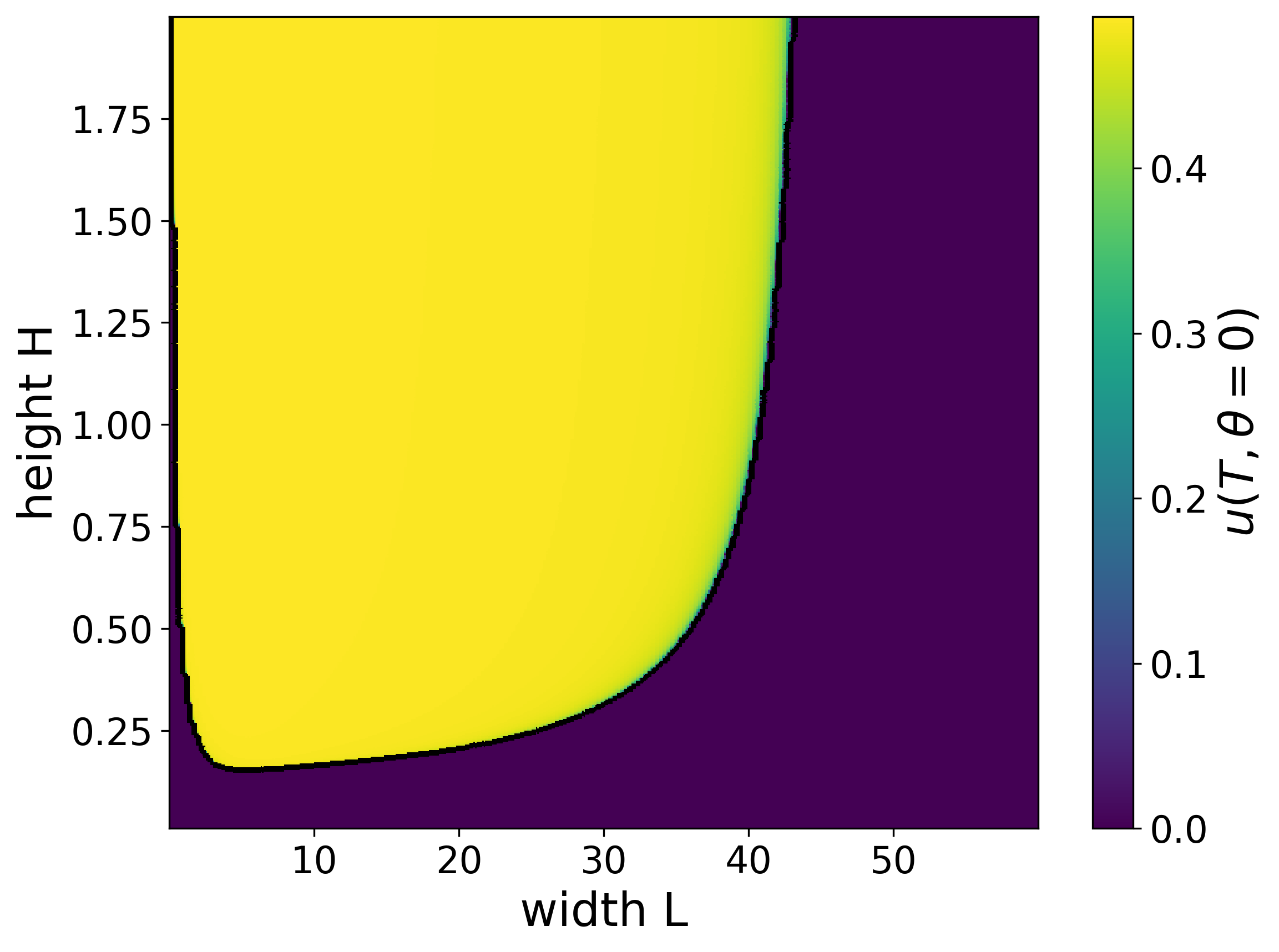}
    \subcaption{$r(\theta) = 2\!\left(2e^{-0.08\theta^{2}}-1\right)$}
    \label{fig:ext4-bounded}
  \end{subfigure}

  \vspace{0.7em}

  \begin{subfigure}[t]{0.48\textwidth}
    \centering
    \includegraphics[width=\linewidth]{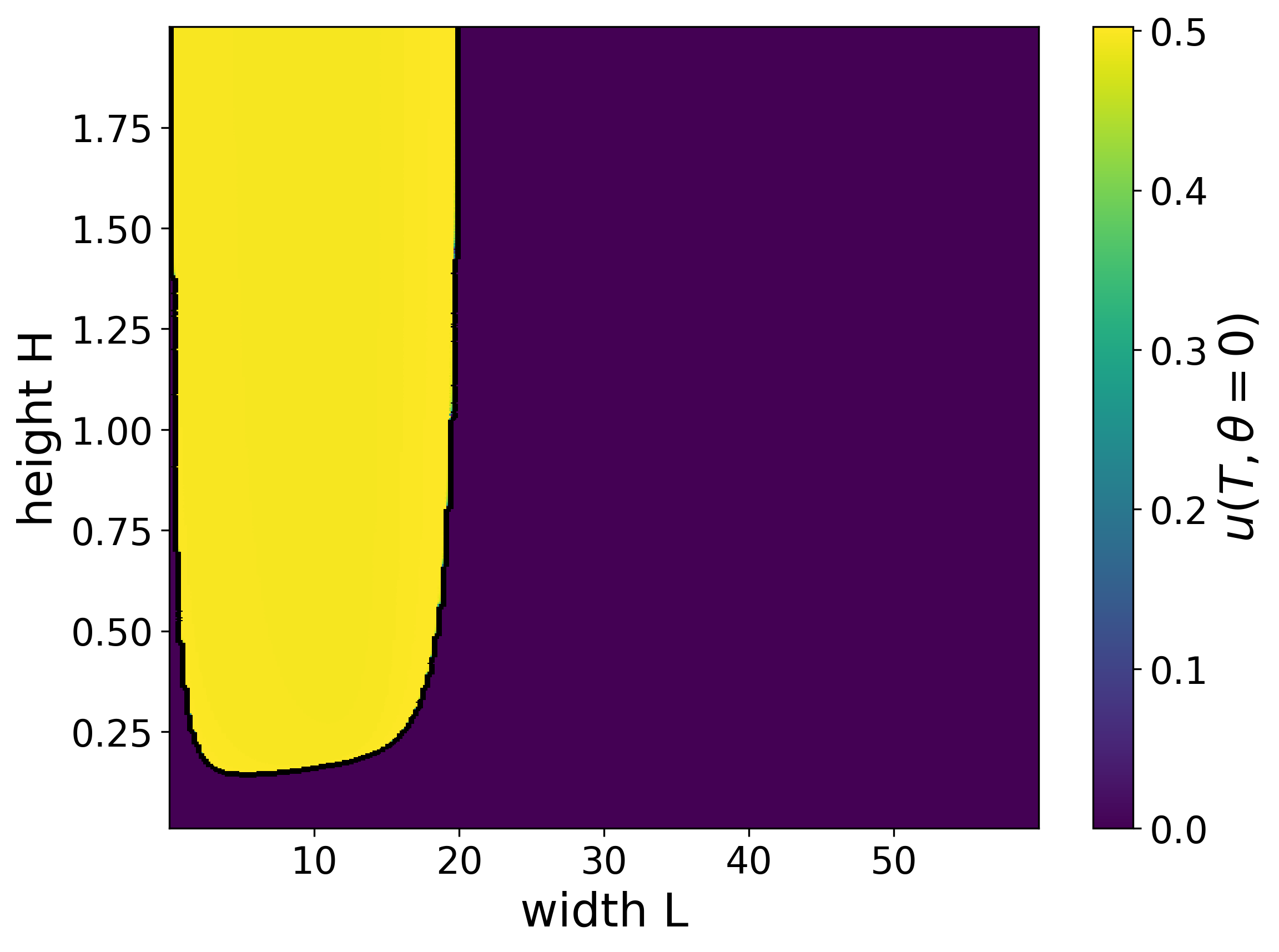}
    \subcaption{$r(\theta) = 2 - 0.01\,\theta^{2}$}
    \label{fig:ext4-small}
  \end{subfigure}
  \hfill
  \begin{subfigure}[t]{0.48\textwidth}
    \centering
    \includegraphics[width=\linewidth]{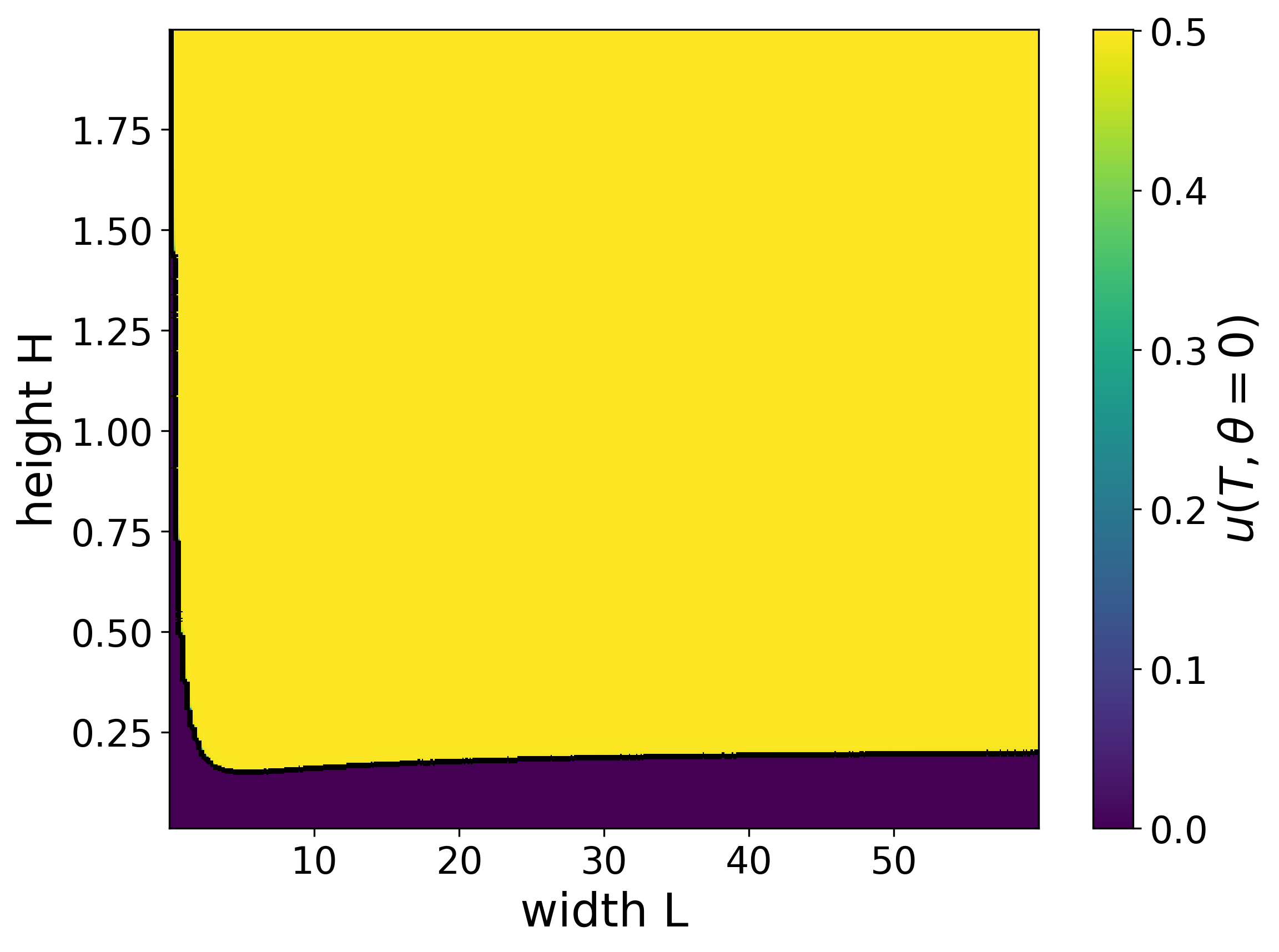}
    \subcaption{$r(\theta) = 2 - 0.2\,\theta^{2}$}
    \label{fig:ext4-strong}
  \end{subfigure}

  \caption{Extinction--persistence diagrams obtained from the final value $u(T,\theta=0)$
  as a function of the initial width $L$ and height $H$ of the initial condition $u_0^{L,H}$.  
  Panels: (a) constant selection; (b) bounded selection profile; 
  (c) weak quadratic decrease; (d) strong quadratic decrease. In all cases, 
  $f(u) = 15u\!\left(1-\tfrac{u}{2\varepsilon}\right)^{2}\mathbf{1}_{\{u<2\varepsilon\}}$
  with $\varepsilon=0.1$ and final time $T=5$.}
  \label{fig:extinction}
\end{figure}

\begin{figure}[htbp]
    \centering
    \includegraphics[width=0.8\textwidth]{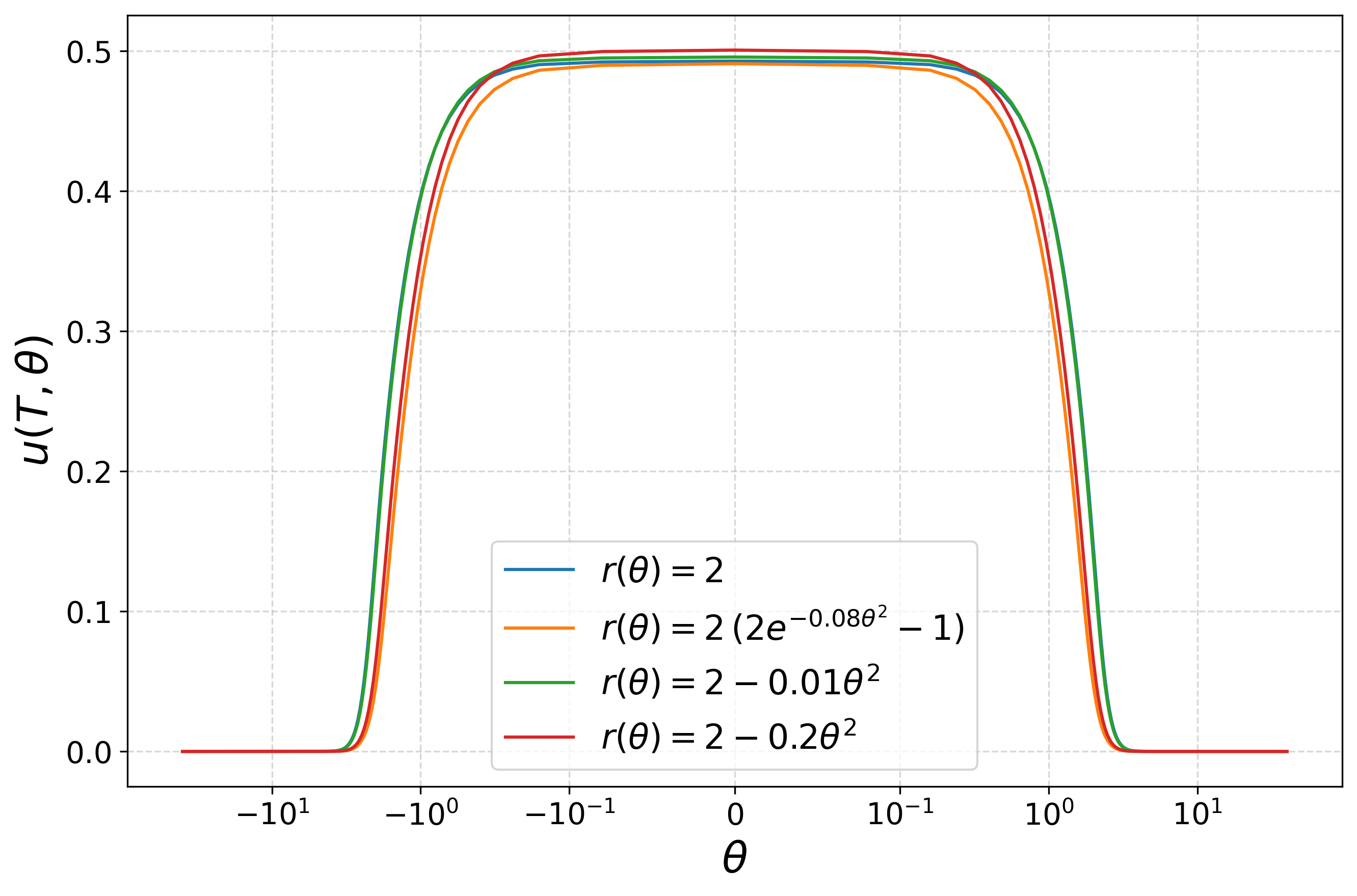}
    \caption{Final-time solution profiles $u(T,\theta)$ for four different selection functions $r$, at $T=5$. 
    Numerically, all the initial conditions that lead to persistence converge to these profiles. 
    A logarithmic scale in $\theta$ is adopted here to highlight the differences between the profiles obtained 
    for the different choices of $r(\theta)$.}
    \label{fig:final_profiles}
\end{figure}





\section{Well-posedness of the Cauchy problem}\label{s:cauchy}

\paragraph*{Proof outline.}
We work in a function space $X_{T^*}$, introduced below in \eqref{def: XT},  and establish a~priori estimates that yield the existence and uniqueness of solutions to the Cauchy problem associated with \eqref{eq:main}; existence is obtained via a fixed-point method and compactness, while uniqueness follows from a comparison estimate.

\smallskip
\emph{(i) Gaussian control, positivity, and pointwise $u$–$\rho_u$ link.}
A heat-kernel comparison (Lemma~\ref{lemma:u<e^(Ct)*gaussian}) yields Gaussian upper bounds for subsolutions of parabolic equations involving quadratically increasing coefficients, as in Assumption~\ref{ass:r}.  Applied to $u^2$, this gives a Gaussian envelope for $u$ on each $[0,T]$ (Proposition~\ref{prop:u<gaussian}). The parabolic maximum principle then ensures $u\ge0$ (and $u>0$ for $t>0$ if $u_0\not\equiv0$), see Proposition~\ref{prop:u positive}. Applying once more the heat-kernel representation, we control $u$ in terms of $\rho_u(t)$ (Lemma~\ref{lemma: u < C rho}).

\smallskip
\emph{(ii) Dynamics of $\rho_u$ via derivative estimates.}
To justify integrating \eqref{eq:main} with respect to $\theta\in\R$ and obtaining an equation on $\rho_u'$, we prove Gaussian decay for $\partial_\theta u$, $\partial_{\theta\theta}u$, and $\partial_t u$ on interior time slabs (Proposition~\ref{prop:du<gaussian}). The argument uses the Duhamel formula, a near/far spatial splitting in the heat-kernel integrals, and the envelope from $(i)$. These controls allow dominated convergence for $\partial_t u$ and show that $\ds\int_\R \partial_{\theta\theta}u\,d\theta=0$, yielding an ODE for $\rho_u$. Then, a logistic-type comparison gives a global bound on $\rho_u$, and combining this with Lemma~\ref{lemma: u < C rho} yields a global bound on $u$ (Proposition~\ref{prop: u < M, rho < M'}).

\smallskip
\emph{(iii) Uniqueness in $X_{T^*}$.}
For two solutions $u,v$, the difference $w=u-v$ solves a linear inequality with bounded coefficients and a Gaussian source depending on $\|\rho_w\|_{L^\infty}$. Lemma~\ref{lemma:u<e^(Ct)*gaussian} gives a short-time estimate forcing $\rho_w\equiv0$, and a continuation argument extends this to $(0,T^*)$ (Proposition~\ref{prop:uniqueness}).

\smallskip
\emph{(iv) Existence by truncation, fixed point, and compactness.}
We first solve a truncated problem with compactly supported fitness term $\varphi \, r$ by a fixed point in a weighted H\"older space, using the Duhamel formula and the bounds from step $(i)$ (Lemma~\ref{lemma:local wel-pos}); continuation yields a global solution (Proposition~\ref{prop:truncated well-pos}). Removing the truncation with cutoffs $\varphi_n\to1$, uniform Gaussian and H\"older controls give equicontinuity; Arzel\`a-Ascoli and dominated convergence allow us to pass to the limit in the Duhamel representation. Parabolic regularity implies that the limit lies in $X_\infty$, and step $(iii)$ implies uniqueness, completing the proof of Theorem~\ref{th:well-pos}.

\subsection{A priori properties}

Let $T^* \in (0,+\infty]$. We define the space
\begin{equation}\label{def: XT}
X_{T^*} := \left\{ u \in C^0([0,T^*) \times \mathbb{R}) \ \middle| \
\begin{aligned}
&u \in C^{1;2}_{t;\theta}((0,T^*) \times \mathbb{R}), \\
&t \mapsto \int_{\mathbb{R}} u(t,\theta)\, d\theta \text{ is locally bounded in } [0,T^*), \\
&\forall T \in [0,T^*), \ \exists C_T, \gamma_T > 0 \ \text{s.t.} \ |u(t,\theta)| \leq C_T e^{\gamma_T \theta^2} \ \text{ in }  [0,T] \times \mathbb{R}
\end{aligned}
\right\}.
\end{equation}

In the sequel, we assume that  $u\in X_{T^*}$ satisfies \eqref{eq:main} for $t \in (0,T^*)$. We derive a priori estimates and properties that will allow us to establish the existence and uniqueness of a solution in $X_{T^*}$. We will begin by establishing the positivity of $u$, and bounding it and its derivatives by a Gaussian function locally in time. This will, in turn, help to prove the well-posedness of the parabolic equation satisfied by $\rho$. Using this equation, we will prove the global boundedness of both $\rho$ and $u$. Before going further on, we establish some properties of the elements of the set $X_{T^*}$.
\begin{lemma} \label{lemma:u<e^(Ct)*gaussian}
Let $T>0$ and $v\in C^0([0,T]\times\R)\cap C^{1;2}_{t;\theta}((0,T]\times\R)$.
Suppose there are $C_0, \mu_0, C_T, \gamma_T > 0$ such that $v(0,\theta) \leq C_0 e^{-\mu_0\theta^2}$ and $|v(t,\theta)| \leq C_T e^{\gamma_T \theta^2} \ \text{ in }  [0,T] \times \R$.

Assume $a \in C^0([0,T] \times \mathbb{R})$ satisfies $a(t,\theta) \leq A$ and $|a(t,\theta)| \leq C_a(1 + \theta^2)$ for some constants $A, C_a > 0$ and all $(t,\theta) \in [0,T] \times \mathbb{R}$.

\begin{enumerate}
    \item Suppose there exist constants $Q_1, \mu_1 > 0$ such that
    $$
    \dt v(t,\theta) - \dthth v(t,\theta)
    \leq a(t,\theta) v(t,\theta) + Q_1 e^{-\mu_1 \theta^2}
    \quad \text{in } (0,T] \times \mathbb{R}.
    $$
    Then, for all $(t,\theta) \in (0,T] \times \mathbb{R}$,
    $$
    v(t,\theta) \leq e^{A t} \left[ C_0 e^{-\frac{\mu_0 \theta^2}{1 + 4 \mu_0 t}} + T Q_1 e^{-\frac{\mu_1 \theta^2}{1 + 4 \mu_1 t}} \right].
    $$

    \item Suppose that $\limsup_{|\theta| \to \infty} v(t,\theta) \leq 0$ uniformly in $t \in [0,T]$, and that
    $$
    \dt v(t,\theta) - \dthth v(t,\theta)
    \leq a(t,\theta) v(t,\theta) + B |v(t,\theta)| + Q_1 e^{-\mu_1 \theta^2}
    \quad \text{in } (0,T] \times \mathbb{R},
    $$
    for some constants $B, Q_1, \mu_1 > 0$. Then, for all $(t,\theta) \in (0,T] \times \mathbb{R}$,
    $$
    v(t,\theta) \leq e^{(A+B)t} \left[ C_0 e^{-\frac{\mu_0 \theta^2}{1 + 4 \mu_0 t}} + T Q_1 e^{-\frac{\mu_1 \theta^2}{1 + 4 \mu_1 t}} \right].
    $$
\end{enumerate}
\end{lemma}

Throughout the proofs of this section, we will note $G(t,\theta):=\frac{1}{\sqrt{4\pi t}} e^{-\frac{\theta^2}{4t}}$ the fundamental solution to the heat equation over $(0,+\infty)\times\R$.

\begin{proof}
    The proof of both results is similar. We begin by defining the function:
    \begin{equation*}
        \bar{v}(t,\theta)
        := e^{Ct} \left[\int_\R G(t,\theta-\eta) C_0 e^{-\mu_0\eta^2}d\eta + \int_0^t \int_\R G(t-s,\theta-\eta)Q_1 e^{-\mu_1\eta^2} \, d\eta \, d s\right],
    \end{equation*}
    where $C$ is taken to be equal to $A$ in the first case of Lemma \ref{lemma:u<e^(Ct)*gaussian} and equal to $A+B$ in the second case.
    We can check that $\bar{v}$ is a non-negative solution of the equation:
    \begin{equation*}
        \dt \bar{v}(t,\theta) - \dthth \bar{v}(t,\theta)
        = C\bar{v}(t,\theta) + Q_1 e^{-\mu_1\theta^2},
    \end{equation*}
    with $\bar{v}(0,\theta)=C_0e^{-\mu_0\theta^2}\geq v(0,\theta)$, for all $\theta \in \R$.
    Furthermore, one may check that, for any $\mu>0$, any $(t,\theta)\in(0,T]\times\R$,
    \begin{equation*}
        \int_\R \frac{1}{\sqrt{4\pi t}}e^{-\frac{(\theta-\eta)^2}{4t}} e^{-\mu\eta^2}d\eta
        = \frac{e^{-\frac{\mu\theta^2}{1+4\mu t}}}{\sqrt{4\pi t}} \int_\R e^{-\frac{1+4\mu t}{4t} (\eta-\frac{\theta}{1+4\mu t})^2} d\eta
        = \frac{e^{-\frac{\mu\theta^2}{1+4\mu t}}}{\sqrt{1+4\mu t}}
        \leq e^{-\frac{\mu\theta^2}{1+4\mu t}},
    \end{equation*}
    and thus, 
    \begin{equation*}
        \bar{v}(t,\theta)
        \leq e^{Ct} \left[
        C_0 e^{-\frac{\mu_0\theta^2}{1+4\mu_0 t}} + Q_1 \int_0^t e^{-\frac{\mu_1\theta^2}{1+4\mu_1 (t-s)}} ds
        \right]
        \leq e^{Ct} \left[
        C_0 e^{-\frac{\mu_0\theta^2}{1+4\mu_0 t}} + Q_1 T e^{-\frac{\mu_1\theta^2}{1+4\mu_1 t}}
        \right].
    \end{equation*}
    All that is left is proving we have $v\leq \bar{v}$. To prove the first item of the lemma, we use the inequality:
    \begin{equation*}
        \dt (v-\bar{v}) - \dthth (v-\bar{v})
        \leq a \, (v-\bar{v}),
        \ \ \ \ \ (t,\theta)\in(0,T]\times\R,
    \end{equation*}
    and, in view of  the hypotheses on $a$ and $v$, we conclude 
thanks  to the    the parabolic maximum principle \cite[Chapter 2, Theorem 9]{friedman-parabolic}. As for  the second item of the lemma, we observe that
    \begin{equation*}
        \dt \bar{v} - \dthth \bar{v}
        \geq a\bar{v} + B\abs{\bar{v}} + Q_1e^{-\mu_1\theta^2},
        \ \ \ \ \ (t,\theta)\in(0,T]\times\R,
    \end{equation*}
    and we conclude thanks to the comparison principle \cite[Chapter 2, Thorem 16]{friedman-parabolic}.
\end{proof}

\begin{prop} \label{prop:u<gaussian}
    Let $T\in(0,T^*)$. There exist $C,\mu>0$ such that
    \begin{equation*}
        |u(t,\theta)| \leq Ce^{-\mu\theta^2}, \quad \text{ for any } (t,\theta)\in[0,T]\times\R.
    \end{equation*}
\end{prop}

In particular, $u(t, \theta) \to 0$ as $|\theta| \to +\infty$, uniformly in time over $[0, T]$.  This will allow us to compare $u$ with functions that vanish at infinity in $\theta$, using the comparison theorem for parabolic partial differential equations. Thanks to this estimate, we also immediately obtain the continuity of $t\mapsto \rho_u(t)$ on $[0,T]$ by applying the dominated convergence theorem.

\begin{proof}
    Since $u\in C^{1;2}_{t,\theta}((0,T^*)\times\R)$, we may consider the parabolic equation satisfied by $u^2$, using the main equation (\ref{eq:main}):
    \begin{equation*}
        \dt (u^2) - \dthth(u^2) = 2u\dt u - 2u\dthth u - 2 (\dth u)^2 \leq 2u(\dt u - \dthth u) = 2r(\theta)u^2 - 2\rho_u(t) u^2 - 2uf(u).
    \end{equation*}
First, we observe that $|u f(u)| \le C_{Lip} u^2$, where $C_{Lip} > 0$ is the Lipschitz constant of $f$, as given by Assumption~\ref{ass:f}. Next, since $u \in X_{T^*}$, the function $\rho_u$ is locally bounded on $[0, T^*)$. Let $M > 0$ be a bound for $\rho_u$ on $[0, T]$. Assumption~\ref{ass:r} further implies that $r$ is bounded from above by some constant $\rmax > 0$. Putting everything together, we obtain:
    \begin{equation*}
        \dt (u^2) - \dthth(u^2) \leq 2(\rmax + M + C_{Lip})u^2.
    \end{equation*}
    We note that, according to Assumption \ref{ass:u0}, we have $u_0^2(\theta) \leq C_0^2 e^{-2\mu_0\theta^2}$.
    Since $u\in X_{T^*}$, $u^2$ verifies the hypotheses of Lemma \ref{lemma:u<e^(Ct)*gaussian}.
    Applying the first case of the lemma to $u^2$, we find $u^2(t,\theta) \leq e^{2(\rmax + M + C_{Lip})T} C_0^2 e^{-\frac{2\mu_0 \theta^2}{1+8 \mu_0 T}}$, and thus:
    \begin{equation*}
        |u(t,\theta)| \leq e^{(\rmax + M + C_{Lip})T} C_0 e^{-\frac{\mu_0 \theta^2}{1+8\mu_0 T}}
    \end{equation*}
    for all $(t,\theta)\in[0,T]\times\R$, which is the estimate we sought.
\end{proof}

\begin{prop} \label{prop:u positive}
    Necessarily, $u\ge 0$. If $u_0\not\equiv 0$, we have $u>0$ in $(0,T^*)\times\R$.
\end{prop}
\begin{proof}
    Let $T\in(0,T^*)$. Set $c(t,\theta):=r(\theta) - \rho_u(t) - g(u(t,\theta))$, with $g(s):=f(s)/s$ for $s\in \R^*$ and $g(0)=f'(0)$. Then, we have $\dt u - \dthth u -c(t,\theta) u =0,$ and, since $u\in X_{T^*}$ and $f$ is Lipschitz-continuous, $c$ is continuous and bounded  in $[0,T] \times \R$. Since $u_0\geq 0$ the parabolic maximum principle implies  $u\ge 0$ in $[0,T] \times \R$. If $u_0 \not \equiv 0$, the strong parabolic maximum principle further implies that $u>0$ in $(0,T] \times \R$.
\end{proof}

\begin{lemma} \label{lemma: u < C rho}
    Let $\tau\in(0,T^*)$. For all $(t,\theta) \in [0, T^*-\tau)\times\R$, we have $u(t+\tau,\theta) \leq \frac{e^{\rmax\tau}}{\sqrt{4\pi\tau}}\rho_u(t)$.
\end{lemma}

\begin{proof}
    Let $\tau\in(0,T^*)$ and $t_0\in[0,T^*-\tau)$. We define $\tilde{u}(t,\theta) := e^{\rmax (t-t_0)} \int_\R G(t-t_0,\theta-\theta') u(t_0,\theta') d\theta'$ for all $(t,\theta)\in(t_0,T^*)\times\R$, which is a solution of the parabolic system:
    \begin{equation*}
        \left\{
        \begin{aligned}
            &\dt \tilde{u} - \dthth\tilde{u} = \rmax \tilde{u},
            \ \ \ \ \ (t,\theta) \in (t_0,T^*)\times\R,
            \\
            &\tilde{u}(t_0,\theta) = u(t_0,\theta),
            \ \ \ \ \ \ \ \  \theta \in \R.
        \end{aligned}
        \right.
    \end{equation*}
    We note that, using the non-negativity of $u$ and $f$ given by Proposition \ref{prop:u positive} and Assumption \ref{ass:f} respectively, we get from the main equation (\ref{eq:main}):
    \begin{equation*}
    		\dt u - \dthth u \leq \rmax u,
            \ \ \ \ \ (t,\theta) \in (t_0,T^*)\times\R
    \end{equation*}
It follows from the comparison principle that $u \leq \tilde{u}$ over $(t_0,T^*)\times\R$.
    From the definitions of $G$ and $\tilde{u}$, we can deduce the following estimate for all $\theta\in\R$:
    \begin{align*}
        u(t_0+\tau,\theta)
        &\leq e^{\rmax\tau} \int_\R G(\tau, \theta-\theta') u(t_0,\theta')d\theta'
      \leq \frac{e^{\rmax\tau}}{\sqrt{4\pi\tau}} \int_\R u(t_0,\theta')d\theta'
        = \frac{e^{\rmax\tau}}{\sqrt{4\pi\tau}} \rho_u(t_0),
    \end{align*}
    which is the estimate we wanted for all $t_0 \in [0,T^*-\tau)$.
\end{proof}

In order to describe the dynamics of $\rho$, if we formally integrate \eqref{eq:main} over $\theta \in \R$, we get:
\begin{equation} \label{eq:rho evolution}
    \rho'(t) = \int_\R r(\theta)u(t,\theta)d\theta - \rho(t)^2 - \int_\R f(u(t,\theta))d\theta,
    \ \ \ \ \ \text{ for all } t>0.
\end{equation}
To justify doing so, we must first prove the following control on the derivatives of $u$ when $\abs{\theta}\to\infty$.

\begin{prop} \label{prop:du<gaussian}
    For every $T_1,T_2\in(0,T^*)$ where $T_1<T_2$, there exist $C,\mu>0$ such that:
    \begin{equation*}
        \abs{\dth u(t,\theta)} + \abs{\dthth u(t,\theta)} + \abs{\dt u(t,\theta)} \leq C e^{-\mu \theta^2},
    \end{equation*}
    for all $(t,\theta) \in [T_1,T_2]\times\R$.
\end{prop}
Thanks to Proposition \ref{prop:du<gaussian}, we can directly deduce (\ref{eq:rho evolution}) by integrating the main equation (\ref{eq:main}) with respect to $\theta$. Indeed, we can apply the dominated convergence theorem to find $\int_\R \dt u(t,\theta)d\theta=\rho'(t)$, as well as the estimate on $\dth u$ to get $\int_\R \dthth u(t,\theta)d\theta = 0$ for all $t\in(0,T^*)$.

\begin{proof}
    Let $T_1,T_2\in(0,T^*)$ such that $T_1<T_2$.
    Define the function $F(t,\theta) := r(\theta)u(t,\theta) - \rho_u(t) u(t,\theta) - f(u(t,\theta))$ for all $(t,\theta)\in[T_1,T_2]\times\R$. Due to the continuity of $\rho_u$, the regularity of $u$, $r$, and $f$, and the non-negativity of $u$ given by Proposition \ref{prop:u positive} (used here, as we only assume that $f\in C^1([0,+\infty))$), we get that $F$ and its derivative $\dth F$ are continuous over $[T_1,T_2]\times\R$.
    $F$ is therefore locally Lipschitz in $\theta\in\R$, uniformly in $t\in[T_1,T_2]$.
    Furthermore, $r$ is bounded by a quadratic function and $u$ is bounded by a Gaussian function in $\theta$ according to Assumption \ref{ass:r} and Proposition \ref{prop:u<gaussian}. Therefore, using the Lipschitz continuity of $f$ and the boundedness of $\rho_u$ over $[T_1,T_2]$, we can bound $F(t,\theta)$ by a Gaussian function $C_F e^{-\mu_F\theta^2}$ for some $C_F,\mu_F>0$ and for all $(t,\theta)\in[T_1,T_2]\times\R$.
    Thanks to existence and uniqueness results of solutions of parabolic differential equations,  \cite[Chapter 1, Theorems 12 and 16]{friedman-parabolic}, we can now assert that the solution $u$ to the problem (\ref{eq:main}) can be expressed as:
    \begin{equation*}
        u(t,\theta) = \int_\R G(t,\theta-\eta) u_0(\eta)d\eta + \int_0^t \int_\R G(t-s,\theta-\eta) F(s,\eta)d\eta ds,
    \end{equation*}
    and its derivative with respect to $\theta$ is given by:
    \begin{equation} \label{eq:du (proof:du<gaussian)}
        \dth u(t,\theta) = \int_\R \dth G(t,\theta-\eta) u_0(\eta)d\eta + \int_0^t \int_\R \dth G(t-s,\theta-\eta) F(s,\eta)d\eta ds,
    \end{equation}
    for all $(t,\theta)\in[T_1,T_2]\times\R$. 
    We note that $\abs{\dth G(t,\theta)} = \frac{\abs{\theta}}{2t\sqrt{4\pi t}} e^{-\frac{\theta^2}{4t}} \leq \frac{1}{t\sqrt{2\pi}}e^{-\frac{\theta^2}{8t}}$, where we bounded $\frac{|\theta|}{\sqrt{8t}}e^{-\frac{\theta^2}{8t}}$ by $1$, since $xe^{-x^2}\leq 1$ for all $x\geq0$.
    We can therefore estimate the first integral of (\ref{eq:du (proof:du<gaussian)}) for all $(t,\theta)\in[T_1,T_2]\times\R$ as follows:
    \begin{align}
        \nonumber
        \abs{\int_\R \dth G(t,\theta-\eta) u_0(\eta)d\eta}
        &\leq \frac{C_0}{t\sqrt{2\pi}} \int_\R e^{-\frac{(\theta-\eta)^2}{8t}} e^{-\mu_0\eta^2} d\eta
        \\
           \label{eq:dG*u0 (proof:du<gaussian)}
        &\leq \frac{2C_0}{\sqrt{t(1+8\mu_0t)}} e^{-\frac{\mu_0\theta^2}{1+8\mu_0t}}
        \leq \frac{2C_0}{\sqrt{T_1(1+8\mu_0T_1)}} e^{-\frac{\mu_0\theta^2}{1+8\mu_0T_2}},
    \end{align}
    where $C_0,\mu_0>0$ are the constants given in Assumption \ref{ass:u0}, 
    and dealing with the convolution of Gaussian functions as we did in the proof of Lemma \ref{lemma:u<e^(Ct)*gaussian}.
    To deal with the second term of (\ref{eq:du (proof:du<gaussian)}), we fix $(t,\theta)\in[T_1,T_2]\times\R$ and split the integral over $\R$ into two parts:
    \begin{align*}
        \int_\R \dth G(t-s,\theta-\eta) F(s,\eta)d\eta
        &= \int_{|\theta-\eta|\leq|\theta|/2} \dth G(t-s,\theta-\eta) F(s,\eta)d\eta
        \\
        &\ \ \ + \int_{|\theta-\eta|>|\theta|/2} \dth G(t-s,\theta-\eta) F(s,\eta)d\eta,
    \end{align*}
    for all $s\in(0,t)$. The condition $|\theta-\eta|\leq \frac{|\theta|}{2}$ implies $|\eta|\geq\frac{|\theta|}{2}$, and we can thus bound the first integral as follows:
    \begin{align}
        \nonumber
        \abs{\int_{|\theta-\eta|\leq|\theta|/2} \dth G(t-s,\theta-\eta) F(s,\eta)d\eta}
        \leq C_F \int_{|\theta-\eta|\leq|\theta|/2} | \dth G(t-s,\theta-\eta) | e^{-\mu_F\eta^2}d\eta
        &
        \\
        \label{eq:theta-eta <= theta/2 (proof:du<gaussian)}
        \leq C_F e^{-\frac{\mu_F \theta^2}{4}} \frac{1}{(t-s)\sqrt{2\pi}} \int_\R e^{-\frac{(\theta-\eta)^2}{8(t-s)}}d\eta
        \leq \frac{2C_F}{\sqrt{t-s}} e^{-\frac{\mu_F \theta^2}{4}}.
        &
    \end{align}
    For the integral over $\{|\theta-\eta|>\frac{|\theta|}{2}\}$, we simply use the fact that $F$ is bounded by $C_F$ over $[T_1,T_2]\times\R$, and we have:
    \begin{align}
        \nonumber
        \abs{\int_{|\theta-\eta|>|\theta|/2} \dth G(t-s,\theta-\eta) F(s,\eta)d\eta}
        \leq \frac{C_F}{(t-s)\sqrt{2\pi}} \int_{|\theta-\eta|>|\theta|/2} e^{-\frac{(\theta-\eta)^2}{8(t-s)}} d\eta
    \end{align}
    Writing the integrand as $e^{-\frac{(\theta-\eta)^2}{8(t-s)}} = e^{-\frac{(\theta-\eta)^2}{16(t-s)}} e^{-\frac{(\theta-\eta)^2}{16(t-s)}}$, we can bound the first term in the product by $e^{-\frac{\theta^2}{64T_2}}$ when $|\theta-\eta|>\frac{|\theta|}{2}$ and $0<s<t\leq T_2$, and bound the integral of the second term by $\int_\R e^{-\frac{(\theta-\eta)^2}{16(t-s)}} d\eta = \sqrt{16(t-s)\pi}$, thus yielding us:
    \begin{equation} \label{eq:theta-eta > theta/2 (proof:du<gaussian)}
        \abs{\int_{|\theta-\eta|>|\theta|/2} \dth G(t-s,\theta-\eta) F(s,\eta)d\eta}
        \leq C_F\sqrt{\frac{8}{t-s}} e^{-\frac{\theta^2}{64T_2}}.
    \end{equation}
    Putting the estimates (\ref{eq:dG*u0 (proof:du<gaussian)}), (\ref{eq:theta-eta <= theta/2 (proof:du<gaussian)}), and (\ref{eq:theta-eta > theta/2 (proof:du<gaussian)}) back into (\ref{eq:du (proof:du<gaussian)}), we get:
    \begin{align*}
        \abs{\dth u(t,\theta)}
        \leq \frac{2C_0}{\sqrt{T_1(1+8\mu_0T_1)}} e^{-\frac{\mu_0\theta^2}{1+8\mu_0T_2}}
        + 4C_F\sqrt{T_2} e^{-\frac{\mu_F \theta^2}{4}}
        + 2C_F\sqrt{8T_2} e^{-\frac{\theta^2}{64T_2}},
    \end{align*}
    for all $(t,\theta)\in[T_1,T_2]\times\R$. Since all three of the terms on the right are Gaussian functions, we can bound all of them by a single Gaussian function $C_1 e^{-\mu_1\theta^2}$ for some $C_1,\mu_1>0$ and for all $\theta\in\R$, hence the inequality of Proposition \ref{prop:du<gaussian} for $\dth u$.

    Let us now prove the result for $\dthth u$. Defining $T_0=T_1/2$, for all $(t,\theta)\in[T_1,T_2]\times\R$, we can express $\dthth u$ as:
    \begin{equation} \label{eq:d2u (proof:du<gaussian)}
        \dthth u(t,\theta) = \int_\R \dthth G(t-T_0,\theta-\eta) u(T_0,\eta)d\eta + \int_{T_0}^t \int_\R \dthth G(t-s,\theta-\eta) F(s,\eta)d\eta ds.
    \end{equation}
    Note that $\abs{\dthth G(t,\theta)} \leq \frac{1}{\sqrt{4\pi t}}(\frac{1}{2t}e^{-\frac{\theta^2}{4t}} + \frac{\theta^2}{4t^2}e^{-\frac{\theta^2}{4t}}) \leq \frac{1}{\sqrt{4\pi t}}(\frac{1}{2t}e^{-\frac{\theta^2}{4t}} + \frac{2}{t}e^{-\frac{\theta^2}{8t}}) \leq \frac{5}{4t^{3/2}\sqrt{\pi}}e^{-\frac{\theta^2}{8t}}$, using the fact that $\frac{\theta^2}{8t}e^{-\frac{\theta^2}{8t}} \leq 1$ because $xe^{-x}\leq 1$ for all $x>0$. Therefore, we have for all $(t,\theta)\in[T_1,T_2]\times\R$:
    \begin{align}
        \nonumber
        \abs{\int_\R \dthth G(t-T_0,\theta-\eta) u(T_0,\eta) d\eta}
        &\leq \frac{5 C_{T_0}}{4(t-T_0)^{3/2}\sqrt{\pi}} \int_\R e^{-\frac{(\theta-\eta)^2}{8(t-T_0)}} e^{-\mu_{T_0} \eta^2} d\eta
        \\
        \label{eq:d2G*u0 (proof:du<gaussian)}
        &\leq \frac{5 C_{T_0}}{T_0\sqrt{2(1+8\mu_{T_0} T_0)}} e^{-\frac{\mu_{T_0} \theta^2}{1+8\mu_{T_0} (T_2-T_0)}},
    \end{align}
    since $t-T_0 \geq T_1-T_0=T_0$, where $C_{T_0},\mu_{T_0}$ are the constants given by Proposition \ref{prop:u<gaussian} such that $u(t,\theta)\leq C_{T_0} e^{-\mu_{T_0}\theta^2}$ for all $(t,\theta)\in [0,T_0]\times\R$.
    For the second term of (\ref{eq:d2u (proof:du<gaussian)}), we fix $(t,\theta)\in[T_1,T_2]\times\R$ and $s\in(T_0,t]$, and we wish to integrate by parts with respect to $\eta$. From the definition of $F$, we have $\dth F(s,\eta)=r'(\eta)u(s,\eta) + r(\eta)\dth u(s,\eta) - \rho_u(s)\dth u(s,\eta) - \dth u(s,\eta) f'(u(s,\eta))$.
    We can apply the previous arguments to the time interval $[T_0,T_2]$ and bound $\dth u(s,\eta)$ by a Gaussian function in $\eta\in\R$, uniformly in $s\in[T_0,T_2]$. Therefore, due to the bounds on $r$, $r'$, and $f'$ given by Assumptions \ref{ass:r} and \ref{ass:f}, we can also bound $\dth F(s,\eta)$ by a Gaussian function $C_F'e^{-\mu_F'\eta^2}$, for some $C_F',\mu_F'>0$ and all $(s,\eta)\in[T_0,T_2]\times\R$.
    We note that we can also bound $\eta \mapsto \dth G(t-s,\theta-\eta)F(s,\eta)$ by such a Gaussian function, and we can thus integrate by parts:
    \begin{equation*}
        \int_\R \dthth G(t-s,\theta-\eta) F(s,\eta)d\eta = - \int_\R \dth G(t-s,\theta-\eta) \dth F(s,\eta)d\eta. 
    \end{equation*}
    Applying the same arguments we did for $\dth u$, we find the bound:
    \begin{equation} \label{eq:dG*dF<gaussian (proof:du<gaussian)}
        \abs{\int_{T_0}^t \int_\R \dth G(t-s,\theta-\eta) \dth F(s,\eta)d\eta d s }
        \leq 4C_F'\sqrt{T_2} e^{-\frac{\mu_F' \theta^2}{4}}
        + 2C_F'\sqrt{8T_2} e^{-\frac{\theta^2}{64T_2}},
    \end{equation}
    for all $(t,\theta)\in[T_1,T_2]\times\R$.
    Inserting the estimates (\ref{eq:d2G*u0 (proof:du<gaussian)}) and (\ref{eq:dG*dF<gaussian (proof:du<gaussian)}) into (\ref{eq:d2u (proof:du<gaussian)}), we can thus bound $\dthth u(t,\theta)$ by a single Gaussian function $C_2e^{-\mu_2\theta^2}$ for some $C_2,\mu_2>0$ and all $(t,\theta)\in[T_1,T_2]\times\R$, and we have proven the result for $\dthth u$.
    
    Finally, to prove the inequality of Proposition \ref{prop:du<gaussian} for $\dt u$, we can simply use the main equation (\ref{eq:main}) to get:
    \begin{equation*}
        \abs{\dt u(t,\theta)}
        \leq \abs{\dthth u(t,\theta)} + \abs{F(t,\theta)} \leq C_2e^{-\mu_2\theta^2} + C_Fe^{-\mu_F\theta^2}
        \leq 2\max(C_2,C_F) e^{-\min(\mu_2,\mu_F)\theta^2},
    \end{equation*}
    for all $(t,\theta)\in[T_1,T_2]\times\R$, thus concluding the proof.
\end{proof}

\begin{prop} \label{prop: u < M, rho < M'}
    Let $M_0,M_0'>0$ such that $u_0(\theta) \leq M_0$ for all $\theta\in\R$ and $\rho_{u_0} \leq M_0'$. There exist $M,M'>0$ such that, for all $(t,\theta) \in [0,T^*)\times\R$:
    \begin{equation*}
        u(t,\theta) \leq M,
        \ \ \ \ \ 
        \rho_u(t) \leq M',
    \end{equation*}
    where we can choose $M'=\max(M_0',\rmax)$, and $M$ depends only on $\rmax$, $M_0$, and $M_0'$.
\end{prop}

\begin{proof}  Thanks to Proposition \ref{prop:du<gaussian}, we can integrate \eqref{eq:main} over $\theta \in \R$ to reach \eqref{eq:rho evolution} and, thus, $\rho_u' \leq \rmax \rho_u - \rho_u^2$ together with $\rho_u(0)=\rho_{u_0}\leq M_0'$. Conclusion $\rho _u\leq M':=\max(M_0',r_{max})$ thus follows from comparison arguments for ODEs.

  Let us now prove the boundedness of $u$. Since $\partial _t u-\partial_{\theta \theta}u \leq r_{max}u$ it follows from the comparison principle that $u(t,\theta)\leq e^{r_{max}t}M_0 \leq e^{\rmax}M_0$ for all $(t,\theta) \in [0,\max\{T^*,1\})\times\R$. If $T^*>1,$  we  use Lemma \ref{lemma: u < C rho} with $\tau = 1$ to get $u(t+1,\theta) \leq \frac{e^{\rmax}}{\sqrt{4\pi}} \rho_u(t) \leq C M'$ for all $(t,\theta)\in(0,T^*-1)\times\R$, where $C$ depends only on $\rmax$. Choosing $M := \max(e^{\rmax}M_0, CM')$, we get the bound on $u$.
\end{proof}

\subsection{Proof of Theorem \ref{th:well-pos}}\label{ss:well-po}

\begin{prop} \label{prop:uniqueness}
 Let $u_0$ satisfy Assumption \ref{ass:u0}. There is at most one solution $u$ to (\ref{eq:main}), starting from $u_0$, such that $u \in X_{T^*}$.
\end{prop}

\begin{proof}
    Let $u,v \in X_{T^*}$ be two solutions to the problem (\ref{eq:main}), and define $w:=u-v$. Then $w$ satisfies the equation $\dt w - \dthth w = r(\theta)w - \rho_u(t)w - \rho_w(t)v - (f(u)-f(v))$ for all $(t,\theta) \in (0,T^*)\times\R$.
    We note $M'$ a global bound of $\rho_u$ and $\rho_v$ given by Proposition \ref{prop: u < M, rho < M'}.
    Let $T\in(0,T^*)$. According to Proposition \ref{prop:u<gaussian}, we can choose $C,\mu>0$ such that $\abs{u(t,\theta)} + \abs{v(t,\theta)} \leq Ce^{-\mu\theta^2}$ for all $(t,\theta)\in[0,T]\times\R$.
    Since $\rho_u$ and $\rho_v$ are both bounded, $\rho_w$ is bounded as well and we have the following inequality:
    \begin{equation} \label{eq:w system (proof:uniqueness)}
        \dt w - \dthth w \leq (r(\theta)-\rho_u(t))w + C_{Lip}|w| + ||\rho_w||_{L^\infty([0,T])} Ce^{-\mu\theta^2},
    \end{equation}
    for all $(t,\theta)\in[0,T]\times\R$.
    We can apply the second case of Lemma \ref{lemma:u<e^(Ct)*gaussian} to $w$ and obtain the control:
    \begin{equation} \label{eq:w<gaussian (proof:uniqueness)}
        |w(t,\theta)| \leq e^{(\rmax+M'+C_{Lip})T} T ||\rho_w||_{L^\infty([0,T])} C e^{-\frac{\mu\theta^2}{1+4\mu T}},
    \end{equation}
    for all $(t,\theta)\in[0,T]\times\R$, as the inequality (\ref{eq:w system (proof:uniqueness)}) remains true if we instead define $w$ to be $v-u$.
    Integrating the previous inequality with respect to $\theta\in\R$, we get:
    \begin{equation*}
        |\rho_w(t)| \leq e^{(\rmax+M'+C_{Lip})T} T ||\rho_w||_{L^\infty([0,T])} C \sqrt{\frac{\pi(1+4\mu T)}{\mu}},
    \end{equation*}
    for all $t\in[0,T]$.
    Defining $q := e^{(\rmax+M'+C_{Lip})T} T C \sqrt{\frac{\pi(1+4\mu T)}{\mu}}$, we can choose $T>0$ small enough to have $0<q<1$.
    Since the previous inequality yields $||\rho_w||_{L^\infty([0,T])} \leq q ||\rho_w||_{L^\infty([0,T])}$, this implies $||\rho_w||_{L^\infty([0,T])}=0$, and thus $w$ is identically $0$ over $[0,T]\times\R$ according to the inequality (\ref{eq:w<gaussian (proof:uniqueness)}).
    
    To conclude the proof, define $\bar{T}\in[0,T^*]$ the maximal time such that $u\equiv v$ over $[0,\bar{T})\times\R$. If we suppose $\bar{T}<T^*$, then by continuity $u(\bar{T},\cdot)\equiv v(\bar{T},\cdot)$ over $\R$. Taking $\bar{T}$ as the initial time, we can apply the arguments we used above to find a time $T>\bar{T}$ such that $u\equiv v$ over $[0,T]\times\R$, which contradicts the definition of $\bar{T}$. Therefore, $\bar{T}=T^*$ and $u\equiv v$ over $[0,T^*)\times\R$, hence the uniqueness of the solution.
\end{proof}

We now turn to the aforementioned truncated problem and consider local solutions.

\begin{lemma} \label{lemma:local wel-pos}
 Let $\varphi\in C^1_c(\R)$ such that $0\leq\varphi\leq1$, $\beta\in(0,1)$, and $C_0,\mu_0,C_\beta>0$. Then there exists $T=T(C_0,\mu_0,C_\beta)>0$ such that, for every $u_0\in C^{0,\beta}(\R)$ verifying
 \begin{equation} \label{eq:local well-pos u0 ass}
0\leq u_0(\theta) \leq C_0 e^{-\mu_0 \theta^2} \quad  \text{ for all } \theta\in\R,
       \quad  \text{ and } \quad \sup_{\substack{\theta_1,\theta_2\in\R,\\ \theta_1\neq\theta_2}} \frac{\abs{u_0(\theta_1)-u_0(\theta_2)}}{\abs{\theta_1-\theta_2}^\beta} \leq C_\beta,
    \end{equation}
    the problem
    \begin{equation} \label{eq:truncated problem local}
    	\begin{cases}
    		\dt u(t,\theta) - \dthth u(t,\theta) = \varphi(\theta) r(\theta) u(t,\theta) - u(t,\theta) \rho_u(t) - f(u(t,\theta)),
          \quad   &(t,\theta) \in (0,T)\times\R,
    		\\
    		u(0,\theta) = u_0(\theta),
            &\theta \in \R,
    	\end{cases}
    \end{equation}
    admits a unique solution $u$ in $X_{T}$.
\end{lemma}

\begin{proof}
    Let us first define for all $T,\kappa,\gamma,\omega>0$ the space:
    \begin{equation*}
        \begin{aligned}
            U_{T,\kappa,\gamma,\omega} :=
            \left\{ u \in C^0([0,T]\times\R), \text{ such that } 0\leq u(t,\theta) \leq \kappa e^{-\gamma \theta^2}\right.\ \ \ &
            \\
            \left.\text{ for all } (t,\theta) \in [0,T]\times\R \text{ and } \abs{u}_\beta\leq\omega\right\},&
        \end{aligned}
    \end{equation*}
    where $\abs{u}_\beta$ is the seminorm defined by:
    \begin{equation*}
        \abs{u}_\beta := \sup_{\substack{(t,\theta_1,\theta_2)\in[0,T]\times\R^2,\\ \theta_1\neq\theta_2,\ \abs{\theta_1-\theta_2}\leq 1}} \frac{\abs{u(t,\theta_1)-u(t,\theta_2)}}{\abs{\theta_1-\theta_2}^\beta}.
    \end{equation*}
    We define $\norm{u}_U := ||u||_\infty + ||u||_1 + \abs{u}_\beta$ its associated norm, where $||u||_\infty = \sup_{(t,\theta)\in[0,T]\times\R}\abs{u(t,\theta)}$ and $||u||_1 = \sup_{t\in[0,T]}\int_\R \abs{u(t,\theta)}d\theta$.
    We note that $U_{T,\kappa,\gamma,\omega}$ is a complete metric space with respect to the norm $||\cdot||_U$.

    Let $u_0 \in C^{0,\beta}(\R)$ a function satisfying the hypothesis (\ref{eq:local well-pos u0 ass}) of Lemma \ref{lemma:local wel-pos}.
    We note $\tilde{f}(x)$ for every $x\geq0$ the function equal to $\frac{f(x)}{x}$ when $x>0$, and to $f'(0)$ when $x=0$. $\tilde{f}$ is therefore continuous over $[0,+\infty)$ and bounded by $C_{Lip}$, and its derivative over $(0,+\infty)$ satisfies, using Taylor's theorem:
    \begin{align*}
        \tilde{f}'(x)
        &= \frac{xf'(x)-f(x)}{x^2}
        = \frac{1}{x^2} \left[ x\left(f'(0)+xf''(0)+\underset{x\to0^+}{o}(x)\right) - xf'(0) - \frac{1}{2} x^2f''(0) + \underset{x\to0^+}{o}(x^2)\right]
        \\
        &= \frac{1}{2}f''(0) + \underset{x\to0^+}{o}(1).
    \end{align*}
    Since $f$ and $f'$ are bounded, $\tilde{f}'$ is thus bounded over $[0,+\infty)$, and $\tilde{f}$ is Lipschitz continuous over $[0,+\infty)$.
    As a result, for every $v\in U_{T,\kappa,\gamma,\omega}$, $\tilde{f}\circ v$ is locally Hölder continuous (of exponent $\beta$) with respect to $\theta$, uniformly in time $t\in[0,T]$.
    Defining for all $v\in U_{T,\kappa,\gamma,\omega}$ and $(t,\theta)\in[0,T]\times\R$ the function $F_v(t,\theta):=\varphi(\theta) r(\theta) - \rho_v(t) - \tilde{f}(v(t,\theta))$, we therefore get that $F_v$ is locally Hölder continuous (of exponent $\beta$) with respect to $\theta\in\R$, uniformly in $t\in[0,T]$, because $r$ and $\varphi$ are continuously differentiable. $F_v$ is also bounded because $\varphi$ has compact support. We also note that $\rho_v$ is continuous due to $v$ being bounded by a Gaussian function in $\theta$ and the dominated convergence theorem.
    Therefore, for every $v\in U_{T,\kappa,\gamma,\omega}$, we can define $\Phi[v]$ to be the unique solution of the linear problem, see \cite[Chapter 1, Theorems 12 and 16]{friedman-parabolic}, 
    \begin{equation} \label{eq:phi[v] system (proof:local well-pos)}
    	\begin{cases}
    		\dt \Phi[v](t,\theta) - \dthth \Phi[v](t,\theta) = F_v(t,\theta)\Phi[v](t,\theta),
            \quad  &(t,\theta) \in (0,T]\times\R
    		\\
    		\Phi[v](0,\theta) = u_0(\theta),  &\theta \in \R,
    	\end{cases}
    \end{equation}
    such that $\Phi[v]\in X_T$.
    Moreover, $\Phi[v]$ is given by the Duhamel formula:
    \begin{equation} \label{eq:phi[v]=G*phi[v] (proof:local well-pos)}
        \Phi[v](t,\theta)
        = \int_\R G(t, \theta - \eta) u_0(\eta)d\eta
        + \int_0^t \int_\R G(t-s,\theta-\eta)
        F_v(s,\eta)\Phi[v](s,\eta) d\eta ds,
    \end{equation}
    for all $(t,\theta)\in(0,T]\times\R$.
    Let us show we can choose $T,\gamma>0$ small enough and $\kappa,\omega>0$ large enough so that $\Phi$ is a contraction over $U_{T,\kappa,\gamma,\omega}$.
    
    We start by showing that $\Phi[v]\in U_{T,\kappa,\gamma,\omega}$.
    First, for all $v\in U_{T,\kappa,\gamma,\omega}$, we have $0 \leq \Phi[v](t,\theta) \leq C_0 e^{\rmax T} e^{-\frac{\mu_0 \theta^2}{1+4T\mu_0}}$ due to arguments showcased in the proofs of Proposition \ref{prop:u positive}, and then applying the first case of Lemma \ref{lemma:u<e^(Ct)*gaussian}.
    Using the same arguments, we can show that any solution $u\in X_T$ to the problem (\ref{eq:truncated problem local}) is non-negative and bounded by the same Gaussian function as $\Phi[v]$.
    We may thus take $T\leq 1$, $\kappa\geq C_0 e^{\rmax}$, and $\gamma \leq \frac{\mu_0}{1+4\mu_0}$, which yields the desired estimate $\Phi[v](t,\theta) \leq \kappa e^{-\gamma\theta^2}$, as well as $u(t,\theta) \leq \kappa e^{-\gamma\theta^2}$ for all $(t,\theta) \in [0,T]\times\R$ and for any solution $u\in X_T$ to \eqref{eq:truncated problem local}.
    We will consider $\kappa$ and $\gamma$ fixed from now on.
    To prove the estimate we want on $|\Phi[v]|_\beta$, we simply write, for all $(t,\theta_1,\theta_2)\in[0,T]\times\R^2$ and using the equation (\ref{eq:phi[v]=G*phi[v] (proof:local well-pos)}):
    \begin{align*}
        \abs{\Phi[v](t,\theta_1)-\Phi[v](t,\theta_2)}
        &= \int_\R G(t, \eta)(u_0(\theta_1-\eta)-u_0(\theta_2-\eta))d\eta
        \\
        &\ \ \ + \int_0^t \int_\R (G(t-s,\theta_1-\eta)-G(t-s,\theta_2-\eta))
        F_v(s,\eta)\Phi[v](s,\eta) d\eta ds.
    \end{align*}
    We bound the first term of the right-hand side by using the Hölder continuity of $u_0$ given by the assumption (\ref{eq:local well-pos u0 ass}), which leads to $\abs{u_0(\theta_1-\eta)-u_0(\theta_2-\eta)}\leq C_\beta\abs{\theta_1-\theta_2}^\beta$.
    For the second term, we use the definition of $F_v$ and the estimate $0\leq v(t,\theta) \leq \kappa e^{-\gamma\theta^2}$ to bound $F_v$ by $C_F := ||\varphi r||_{L^\infty(\R)} + \kappa\sqrt{\frac{\pi}{\gamma}} + C_{Lip}$.
    From the Hölder continuity of the function $x\in\R\mapsto \frac{1}{\sqrt{4\pi}}e^{-x^2}$, we can write $\abs{G(t-s,\theta_1-\eta)-G(t-s,\theta_2-\eta)} \leq \frac{C_\beta'}{\sqrt{t-s}} \frac{1}{(t-s)^{\beta/2}} \abs{\theta_1-\theta_2}^\beta$, for some $C_\beta'>0$ which depends only on $\beta$.
    Supposing $T\leq1$ without loss of generality, and using the bound $|\Phi[v]|\leq \kappa$, we therefore obtain the estimate:
    \begin{equation*}
        \frac{\abs{\Phi[v](t,\theta_1)-\Phi[v](t,\theta_2)}}{\abs{\theta_1-\theta_2}^\beta}
        \leq C_\beta + C_\beta' C_F \kappa \frac{2}{1-\beta} t^{\frac{1-\beta}{2}}
        \leq C_\beta + \frac{2C_\beta' C_F \kappa}{1-\beta},
    \end{equation*}
    for all $(t,\theta_1,\theta_2)\in[0,T]\times\R^2$ such that $\theta_1\neq\theta_2$.
    Fixing $\omega>0$ to be larger than the right-hand side of the above inequality, we thus have $\abs{\Phi[v]}_\beta \leq \omega$ and $\Phi[v] \in U_{T,\kappa,\gamma,\omega}$. We can again use the same arguments on any solution $u\in X_T$ to the problem (\ref{eq:truncated problem local}) to show that $|u|_\beta$ has the same bound as $|\Phi[v]|_\beta$, and therefore $u$ necessarily belongs to $U_{T,\kappa,\gamma,\omega}$.
    
    Now that we have established that $\Phi$ maps $U_{T,\kappa,\gamma,\omega}$ into itself, let us show that it is a contraction with respect to $\norm{\cdot}_U$. Let $v,w\in U_{T,\kappa,\gamma,\omega}$. From the system (\ref{eq:phi[v] system (proof:local well-pos)}) satisfied by $\Phi[v]$ and $\Phi[w]$, $\Phi[v]-\Phi[w]$ satisfies:
    \begin{equation} \label{eq:phi[v]-phi[w] system (proof:local well-pos)}
    	\left\{
    	\begin{aligned}
    		&\dt (\Phi[v]-\Phi[w]) - \dthth (\Phi[v]-\Phi[w]) = F_v(t,\theta)(\Phi[v]-\Phi[w]) + (F_v(t,\theta)-F_w(t,\theta))\Phi[w],
    		\\
    		&(\Phi[v]-\Phi[w])(0,\theta) = 0,
    	\end{aligned}
    	\right.
    \end{equation}
    for all $(t,\theta)\in(0,T]\times\R$.
    From the definition of $F_v$ and $F_w$, we have $\abs{F_v(t,\theta)-F_w(t,\theta)} = \abs{\rho_{v-w} + \tilde{f}(v(t,\theta))-\tilde{f}(w(t,\theta))} \leq ||v-w||_1 + \tilde{C}_{Lip}||v-w||_\infty$, where $\tilde{C}_{Lip}>0$ is the Lipschitz constant of $\tilde{f}$. Moreover, $\Phi[v]\in U_{T,\kappa,\gamma,\omega}$ and so we can bound it by $\kappa e^{-\gamma\theta^2}$ and get:
    \begin{equation*}
        \dt (\Phi[v]-\Phi[w]) - \dthth (\Phi[v]-\Phi[w])
        \leq C_F\abs{\Phi[v]-\Phi[w]} + \kappa e^{-\gamma\theta^2} \left(||v-w||_1 + \tilde{C}_{Lip}||v-w||_\infty\right),
    \end{equation*}
    Applying the second case of Lemma \ref{lemma:u<e^(Ct)*gaussian} to $\Phi[v]-\Phi[w]$, we get the estimate:
    \begin{equation*}
        \abs{\Phi[v](t,\theta)-\Phi[w](t,\theta)} \leq \kappa T e^{C_F T} e^{-\frac{\gamma\theta^2}{1+4T}} \left(||v-w||_1 + \tilde{C}_{Lip}||v-w||_\infty\right),
    \end{equation*}
    for all $(t,\theta)\in[0,T]\times\R$, after noting that the same arguments apply to $\Phi[w]-\Phi[v]$.
    Choosing $T>0$ small enough (independently of $v$ and $w$), we can thus obtain the estimates $||\Phi[v]-\Phi[w]||_\infty \leq \frac{1}{4}(||v-w||_1 + ||v-w||_\infty)$ and $||\Phi[v]-\Phi[w]||_1 \leq \frac{1}{4}(||v-w||_1 + ||v-w||_\infty)$.
    We now prove that we can get a similar estimate for $|\Phi[v]-\Phi[w]|_\beta$.
    For all $(t,\theta_1,\theta_2) \in (0,T]\times\R^2$, we have:
    \begin{align*}
        &\abs{(\Phi[v]-\Phi[w])(t,\theta_1) - (\Phi[v]-\Phi[w])(t,\theta_2)}
        \leq \int_0^t \int_\R (G(t-s,\theta_1-\eta)-G(t-s,\theta_2-\eta))\times
        \\
        &\qquad\qquad\qquad\qquad\qquad      \Big[ F_v(s,\eta)(\Phi[v](s,\eta)-\Phi[w](s,\eta))+ (F_v(s,\eta)-F_w(s,\eta))\Phi[w](s,\eta) \Big] d\eta ds.
    \end{align*}
    As above, the Hölder continuity of $x\in\R\mapsto \frac{1}{\sqrt{4\pi}}e^{-x^2}$ yields $\abs{G(t-s,\theta_1-\eta)-G(t-s,\theta_2-\eta)} \leq \frac{C_\beta'}{(t-s)^{(1+\beta)/2}} \abs{\theta_1-\theta_2}^\beta$.
    For the other terms inside the integral, we have the bounds:
    \begin{align*}
        \abs{\int_\R F_v(s,\eta)(\Phi[v](s,\eta)-\Phi[w](s,\eta)) d\eta} \leq C_F||\Phi[v]-\Phi[w]||_1,
    \end{align*}
    and,
    \begin{align*}
        \abs{\int_\R (F_v(s,\eta)-F_w(s,\eta))\Phi[w](s,\eta) d\eta}
        &\leq (||v-w||_1 + \tilde{C}_{Lip}||v-w||_\infty)\int_{\R} \kappa e^{-\gamma\eta^2} d\eta
        \\
        &\leq \kappa\sqrt{\frac{\pi}{\gamma}}(||v-w||_1 + \tilde{C}_{Lip}||v-w||_\infty).
    \end{align*}
    Putting these estimates together, we have:
    \begin{align*}
        &\frac{\abs{(\Phi[v]-\Phi[w])(t,\theta_1) - (\Phi[v]-\Phi[w])(t,\theta_2)}}{\abs{\theta_1-\theta_2}^\beta}
        \\
        &\ \ \ \ \ \ \ \ \ \ \ \ \ \ \ \ \ \ \ \ \ \ 
        \leq \frac{2C_\beta'}{1-\beta} T^{\frac{1-\beta}{2}} \left[ C_F ||\Phi[v]-\Phi[w]||_1 + \kappa\sqrt{\frac{\pi}{\gamma}}(||v-w||_1 + \tilde{C}_{Lip}||v-w||_\infty)\right]
        \\
        &\ \ \ \ \ \ \ \ \ \ \ \ \ \ \ \ \ \ \ \ \ \ 
        \leq \frac{2C_\beta'}{1-\beta} T^{\frac{1-\beta}{2}} \left[(\frac{C_F}{4}+\kappa\sqrt{\frac{\pi}{\gamma}})||v-w||_1 + (\frac{C_F}{4}+\tilde{C}_{Lip})||v-w||_\infty\right],
    \end{align*}
    for all $(t,\theta_1,\theta_2)\in[0,T]\times\R^2$ such that $\theta_1\neq\theta_2$.
    We can thus choose $T>0$ small enough (independently of $v$ and $w$) such that $\abs{\Phi[v]-\Phi[w]}_\beta \leq \frac{1}{4}(||v-w||_1 + ||v-w||_\infty)$.
    Finally, adding up our estimates, we get $||\Phi[v]-\Phi[w]||_U \leq \frac{3}{4}||v-w||_U$, making $\Phi$ a contraction of $U_{T,\kappa,\gamma,\omega}$.

    Applying the Banach fixed-point theorem, we know $\Phi$ has a unique fixed point in $U_{T,\kappa,\gamma,\omega}$, corresponding to a unique solution to the problem (\ref{eq:truncated problem local}). Since we have shown that any solution $u\in X_T$ necessarily belongs to $U_{T,\kappa,\gamma,\omega}$, the solution is unique in $X_T$ as well, which concludes the proof of Lemma \ref{lemma:local wel-pos}.
\end{proof}

We now turn to global solutions for the truncated problem. Recall that, for any $T^* \in (0,+\infty]$, $X_{T^*}$ was defined in \eqref{def: XT}.

\begin{prop} \label{prop:truncated well-pos}
    Let $\varphi\in C^1_c(\R)$   such that $0\leq \varphi \leq 1$. Then there exists a unique global solution $u$ to the problem
    \begin{equation} \label{eq:truncated problem}
    	\begin{cases}
    		\dt u(t,\theta) - \dthth u(t,\theta) = \varphi(\theta) r(\theta) u(t,\theta) - u(t,\theta) \rho_u(t) - f(u(t,\theta)),
            \quad & (t,\theta) \in (0,+\infty)\times\R,
    		\\
    		u(0,\theta) = u_0(\theta),
            \quad  & \theta \in \R,
    \end{cases}
    \end{equation}
    such that $u\in X_{\infty}$. Furthermore, $u$ s given by
    \begin{align}
        \nonumber
        u(t,\theta)
        = &\int_\R G(t,\theta-\eta)u_0(\eta)d\eta
        \\
        \label{eq:u=G*phi*u}
        &+ \int_0^t \int_\R G(t-s,\theta-\eta) \left[\varphi(\eta) r(\eta) u(s,\eta) - u(s,\eta) \rho_u(s) - f(u(s,\eta))\right] d\eta ds,
    \end{align}
    for all $(t,\theta)\in(0,+\infty)\times\R$.
\end{prop}

\begin{proof}
    Note $\bar{T}$ the maximal time such that there exists a unique solution $u\in X_{\bar{T}}$ to the problem (\ref{eq:truncated problem}) over $[0,\bar{T})\times\R$, and suppose $\bar{T}<+\infty$.
    We know from Lemma \ref{lemma:local wel-pos} that $\bar{T}>0$.
    Applying Proposition \ref{prop:u positive} and Lemma \ref{lemma:u<e^(Ct)*gaussian}, we have for all $(t,\theta)\in[0,\bar{T})\times\R$ the inequalities $0\leq u(t,\theta)\leq C_0 e^{\rmax \bar{T}} e^{-\frac{\mu_0 \theta^2}{1+4\bar{T}\mu_0}}$, where $C_0,\mu_0>0$ are the constants defined in Assumption \ref{ass:u0}.
    Moreover, in the proof of Lemma \ref{lemma:local wel-pos}, we also show the inequality:
    \begin{equation*}
        \frac{\abs{u(t,\theta_1)-u(t,\theta_2)}}{\abs{\theta_1-\theta_2}^\beta}
        \leq C_\beta + C_\beta' C_F \frac{2}{1-\beta} \bar{T}^{\frac{1-\beta}{2}}
    \end{equation*}
    for all $(t,\theta_1,\theta_2)\in[0,\bar{T})\times\R^2$ such that $\theta_1\neq\theta_2$, where $\beta\in(0,1)$ is the Hölder exponent of $u_0$ given by Assumption \ref{ass:u0}, $C_\beta>0$ is the $\beta$-Hölder coefficient of $u_0$, $C_\beta'$ is the $\beta$-Hölder coefficient of $x\in\R\mapsto \frac{1}{\sqrt{4\pi}}e^{-x^2}$, and $C_F := ||\varphi r||_{L^\infty(\R)} + C_0 e^{\rmax \bar{T}}\sqrt{\frac{\pi(1+4\bar{T}\mu_0)}{\mu_0}}+C_{Lip}$ is a bound of $(t,\theta)\mapsto\varphi(\theta)r(\theta)-\rho_u(t)-\frac{f(u(t,\theta))}{u(t,\theta)}$ over $[0,\bar{T})\times\R$.
    Finally, this shows that there exist  $\bar{C}_0$, $\bar{\mu}_0$, $\bar{C}_\beta,$ such that, for all $\tau\in (0,\bar{T})$, we have $0\le u(\tau,\theta) \le \bar{C}_0 e^{-\bar{\mu}_0 \theta^2}$ and
    $$\sup_{\substack{\theta_1,\theta_2\in\R,\\ \theta_1\neq\theta_2}} \frac{\abs{u(\tau,\theta_1)-u(\tau,\theta_2)}}{\abs{\theta_1-\theta_2}^\beta} \leq \bar{C}_\beta.$$
    Thus according to Lemma \ref{lemma:local wel-pos}, there exists $T(\bar{C}_0,\bar{\mu}_0,\bar{C}_\beta)>0$ such that, for all $\tau\in[0,\bar{T})$, the local problem (\ref{eq:truncated problem local}) with initial condition $u(\tau,\cdot)$ has a unique solution $\tilde{u}$ over $[\tau,\tau+T]\times\R$ such that $\tilde{u}(\cdot-\tau,\cdot)\in X_T$.
    Taking $\tau = \bar{T}-T/2$, we can therefore extend the unique solution to the problem \eqref{eq:truncated problem} to the domain $[0,\bar{T}+T/2)\times\R$, which contradicts the definition of $\bar{T}$.
    Thus, $\bar{T}=+\infty$ which proves the existence of a unique solution over $[0,+\infty)\times\R$.

    Since, for all $T>0$, the function $(t,\theta)\mapsto\varphi(\theta)r(\theta) u(t,\theta) - u(t,\theta) \rho_u(t) - f(u(t,\theta))$ is bounded over $[0,T]\times\R$ and locally Hölder-continuous in $\theta\in\R$, uniformly in $t\in[0,T]$, we can verify that the right-hand side of the equation (\ref{eq:u=G*phi*u}) is also a solution to the problem (\ref{eq:truncated problem}) and belongs to $X_\infty$. By uniqueness, we get that $u$ satisfies \eqref{eq:u=G*phi*u}.
\end{proof}

We are now ready to prove Theorem \ref{th:well-pos}, that is, the existence and uniqueness of a solution to the main problem (\ref{eq:main}).

\begin{proof}[Proof of Theorem \ref{th:well-pos}]
Let $\varphi\in C_c^1(\R)$ with $0\le \varphi\le 1$, $\varphi\equiv1$ on $[-1,1]$ and $\varphi\equiv0$ on $\R\setminus(-2,2)$, and set $\varphi_n(\theta):=\varphi(\theta/n)$, $n\in\N^*$. By Proposition \ref{prop:truncated well-pos}, for each $n$ there exists a unique global solution $u_n\in X_\infty$ to
\[
\partial_t u_n-\partial_{\theta\theta} u_n
= \varphi_n(\theta) r(\theta) u_n - \rho_{u_n}(t)\,u_n - f(u_n),\qquad
u_n(0,\theta)=u_0(\theta),
\]
with the representation
\begin{equation}\label{eq:mild-n}
u_n(t,\theta)=G(t)*u_0(\theta)+ \int_0^t\!\!\int_\R G(t-s,\theta-\eta)\Big(\varphi_n r\,u_n-\rho_{u_n}u_n-f(u_n)\Big)(s,\eta)\,d\eta\,ds.
\end{equation}

\medskip
\noindent\textbf{Step 1: Uniform bounds (independent of $n$).}
For any fixed $T>0$, Proposition \ref{prop:u positive} and Lemma \ref{lemma:u<e^(Ct)*gaussian} give a Gaussian envelope
\begin{equation}\label{eq:env}
0\le u_n(t,\theta)\le \kappa\, e^{-\gamma\theta^2}\qquad\text{for all }(t,\theta)\in[0,T]\times\R,
\end{equation}
with $\kappa=\kappa(T)>0$ and $\gamma=\gamma(T)>0$ independent of $n$. In particular,
\[
\|u_n\|_{L^\infty([0,T]\times\R)}\le \kappa,\qquad
\sup_{t\in[0,T]}\int_\R u_n(t,\theta)\,d\theta\le \kappa\sqrt{\frac{\pi}{\gamma}},
\]
hence $|\rho_{u_n}(t)|\le M:=\kappa\sqrt{\pi/\gamma}$ on $[0,T]$, uniformly in $n$.
By the growth assumption on $r$ and the boundedness of $f$, the term
\[
H_n(t,\theta):=\varphi_n(\theta)r(\theta)u_n(t,\theta)-\rho_{u_n}(t)u_n(t,\theta)-f(u_n(t,\theta))
\]
satisfies the uniform bound
\begin{equation}\label{eq:F-unif}
|H_n(t,\theta)|\le C_\star\qquad\text{for all }(t,\theta)\in[0,T]\times\R,
\end{equation}
for some $C_\star=C_\star(T)$ independent of $n$.

\medskip

\noindent\textbf{Step 2: Equicontinuity.}

\smallskip
\noindent\emph{Uniform equicontinuity in $\theta$.}
Using the representation formula \eqref{eq:mild-n} with 
$\theta_1,\theta_2\in\mathbb{R}$, write
\[
|u_n(t,\theta_1)-u_n(t,\theta_2)| \;\le\; A+B,
\]
with
\[
A:=\Big|\int_{\mathbb{R}}G(t,\eta)\big(u_0(\theta_1-\eta)-u_0(\theta_2-\eta)\big)\,d\eta\Big|,
\quad
B:=\Big|\int_0^t\!\!\int_{\mathbb{R}}\!\big(G_{1}-G_{2}\big)H_n\,d\eta\,ds\Big|,
\]
where $G_{1}=G(t-s,\theta_1-\eta)$, $G_{2}=G(t-s,\theta_2-\eta)$ and $H_n=H_n(s,\eta)$.

We first deal with the term $A$. Since $[u_0]_{C^{0,\beta}(\R)}:=\sup_{x\neq y}\frac{|u_0(x)-u_0(y)|}{|x-y|^\beta}<+\infty$, we have
\[
A \le [u_0]_{C^{0,\beta}(\mathbb{R})}\,|\theta_1-\theta_2|^{\beta}\int_{\mathbb{R}}G(t,\eta)\,d\eta
= [u_0]_{C^{0,\beta}(\mathbb{R})}\,|\theta_1-\theta_2|^{\beta}.
\]

Next, we deal with the term $B$.  Let us recall that, for any $g\in W^{1,1}(\R)$, any $h\in\R$, there holds
\begin{equation}\label{W11}
\int_{\R}\big|g(x+h)-g(x)\big|\,dx \;\le\; |h|\,\|g'\|_{L^1(\R)}.
\end{equation}
Since
\[
\partial_\eta G(\tau,\eta)= -\frac{\eta}{2\tau}G(\tau,\eta),\qquad
\|\partial_\eta G(\tau,\cdot)\|_{L^1(\mathbb{R})}
=\frac{1}{\sqrt{\pi\,\tau}},
\]
using \eqref{W11} with  $g(\cdot)=G(t-s,\cdot)$ and $h=\theta_1-\theta_2$, and changing variables
$y=\theta_2-\eta$, we obtain
\begin{align*}
    \int_\R |G_{1}-G_{2}| d\eta
 \le |\theta_1-\theta_2|\,\|\partial_\eta G(t-s,\cdot)\|_{L^1(\R)} 
= \frac{|\theta_1-\theta_2|}{\sqrt{\pi\,(t-s)}}.
\end{align*}
With $\|H_n\|_{L^\infty([0,T]\times\mathbb{R})}\le C_\star$  from \eqref{eq:F-unif},
\[
B \le C_\star\,|\theta_1-\theta_2| \int_0^t \frac{ds}{\sqrt{\pi\,(t-s)}}
= \frac{2C_\star}{\sqrt{\pi}}\,\sqrt{t}\,|\theta_1-\theta_2|
\le \frac{2C_\star}{\sqrt{\pi}}\,\sqrt{T}\,|\theta_1-\theta_2|.
\]

\smallskip
Finally, for all $t\in[0,T]$ and $\theta_1,\theta_2\in\mathbb{R}$,
\[
|u_n(t,\theta_1)-u_n(t,\theta_2)|
\;\le\; [u_0]_{C^{0,\beta}(\mathbb{R})}\,|\theta_1-\theta_2|^{\beta}
\;+\; \frac{2C_\star}{\sqrt{\pi}}\,\sqrt{T}\,|\theta_1-\theta_2|.
\]
Thus, $(u_n)$ is uniformly equicontinuous in $\theta$ on $[0,T]\times\mathbb{R}$. The above formula also gives  the spatial Hölder control
\begin{equation}\label{eq:holder-1}
\sup_{\substack{t\in[0,T],\,\theta_1,\theta_2\in\R\\ 0<|\theta_1-\theta_2|\le 1}}
\frac{|u_n(t,\theta_1)-u_n(t,\theta_2)|}{|\theta_1-\theta_2|^\beta}\ \le\ C_\star',
\end{equation}
for some constant $C_\star'$ independent of $n$.

\medskip

\noindent\emph{Equicontinuity in $t$.}
Fix $0\le s<t\le T$ and $\theta\in\R$ and set $h:=t-s$. Writing
\begin{equation}\label{eq:duhamel_tau}
u_n(t,\theta)
= \big(G(h)*u_n(s,\cdot)\big)(\theta)
+\int_s^t\!\!\int_{\R} G(t-\tau,\theta-\eta)\,H_n(\tau,\eta)\,d\eta\,d\tau,
\end{equation}
we get
\[
|u_n(t,\theta)-u_n(s,\theta)| \;\le\; I_1+I_2,
\]
with
\[
I_1:=\int_{\R} G(h,\eta)\,\big|u_n(s,\theta-\eta)-u_n(s,\theta)\big|\,d\eta,\qquad
I_2:=\int_s^t \|H_n(\tau,\cdot)\|_{L^\infty}\,d\tau.
\]

We split the integral $I_1$ into $\{|\eta|\le 1\}$ and $\{|\eta|>1\}$ and write, with obvious notations, $I_1 = I_{1,\mathrm{small}}+I_{1,\mathrm{tail}}$. Using \eqref{eq:holder-1}, we get
\[
I_{1,\mathrm{small}}
\le C_\star' \int_{|\eta|\le 1} G(h,\eta)\,|\eta|^\beta\,d\eta
\le C_\star' \int_{\R} G(h,\eta)\,|\eta|^\beta\,d\eta.
\]
Moreover,
\begin{align}
\int_\R G(h,\eta)\,|\eta|^\beta\,d\eta
&=\frac{1}{\sqrt{4\pi h}}\int_\R e^{-\frac{\eta^2}{4h}}|\eta|^\beta\,d\eta
=\frac{1}{\sqrt{4\pi h}}(2\sqrt{h})^{\beta+1}\int_\R e^{-y^2}|y|^\beta\,dy \nonumber \\
&=\frac{1}{\sqrt{\pi}}\,2^{\beta}h^{\frac{\beta}{2}}
\Big(2\int_0^{+\infty} e^{-y^2}y^\beta\,dy\Big)
=\frac{2^{\beta}}{\sqrt{\pi}}\,\Gamma\!\Big(\frac{\beta+1}{2}\Big)\,h^{\frac \beta 2}. \label{eq:G(h)}
\end{align}
Hence $I_{1,\mathrm{small}} \le C_\star'\, h^{\frac \beta2}$ for some constant $C_\star'$ which does not depend on $n$ (we do not change the constant name for simplicity). For the tail, use the estimate of Step~1,  $\|u_n\|_{L^\infty}\le \kappa$ and~\eqref{eq:G(h)}:
\[I_{1,\mathrm{tail}}
\;\le\; 2\kappa \int_{|\eta|>1} G(h,\eta)\,d\eta
\;\le\; 2\kappa  \int_{\R} |\eta|^\beta G(h,\eta)\,d\eta
= 2\kappa \frac{2^{\beta}}{\sqrt{\pi}}\,\Gamma\!\Big(\frac{\beta+1}{2}\Big)\,h^{\frac \beta 2}.\]
Thus, in all cases,
\[
I_1 \le C_\star'\, h^{\frac \beta2}.
\]
Using $\|H_n\|_{L^\infty([0,T]\times\mathbb{R})}\le C_\star$, we readily get $I_2 \le C_\star |t-s|$. Finally, for all $s,t\in[0,T]$ and $\theta\in\R$,
\[
|u_n(t,\theta)-u_n(s,\theta)|
\;\le\; I_1+I_2 \le C_\star' |t-s|^{\beta/2}+C_\star\,|t-s|.
\]

Thus, for each $R>0$, $(u_n)$ is equicontinuous and uniformly bounded on each compact strip
\[
K_R:=[0,T]\times[-R,R].
\]

\medskip

\noindent\textbf{Step 3: Arzelà-Ascoli and diagonal extraction.}
By Arzelà-Ascoli, for each $R$ there exists a subsequence (not relabeled) converging uniformly on $K_R$.
A diagonal argument provides a single subsequence (still denoted $u_n$) and a limit $u\in C^0([0,T]\times\R)$ such that
\[
u_n\to u\quad\text{uniformly on }K_R\ \text{for every }R>0.
\]

\medskip

\noindent\textbf{Step 4: Global uniform convergence via Gaussian tails.}
The envelope \eqref{eq:env} passes to the limit, hence $0\le u(t,\theta)\le \kappa e^{-\gamma\theta^2}$ on $[0,T]\times\R$.
Given $\varepsilon>0$, pick $R$ so large that $2\kappa e^{-\gamma R^2}<\varepsilon/2$. Then, for $n$ large enough,
\[
\sup_{[0,T]\times\R}|u_n-u|
\le \sup_{K_R}|u_n-u|+\sup_{|\theta|>R}(u_n+u)
\le \varepsilon/2+2\kappa e^{-\gamma R^2}<\varepsilon.
\]
Thus $u_n\to u$ uniformly on $[0,T]\times\R$.

\medskip
\noindent\textbf{Step 5: Passage to the limit in the equation.}
Uniform convergence implies $f(u_n)\to f(u)$ uniformly. Moreover, by the Gaussian bound
$u_n\to u$ in $L^1(\R)$ uniformly in $t\in[0,T]$, hence $\rho_{u_n}(t)=\int_\R u_n(t,\theta)\,d\theta\to \rho_u(t)$ uniformly in $t$.
In the  formula \eqref{eq:mild-n}, the integrands are uniformly bounded by an integrable function
(heat kernel times the Gaussian envelope), so dominated convergence yields, for each $(t,\theta)\in(0,T]\times\R$,
\begin{equation}\label{eq:mild-limit}
u(t,\theta)=G(t)*u_0(\theta) + \int_0^t\!\!\int_\R G(t-s,\theta-\eta)\Big(r(\eta)u-\rho_u(s)u-f(u)\Big)(s,\eta)\,d\eta\,ds.
\end{equation}
By classical parabolic regularity (e.g., \cite[Chapter1, Theorem 12]{friedman-parabolic}) applied to \eqref{eq:mild-limit}, $u\in X_T$ and solves
\[
\partial_t u-\partial_{\theta\theta}u=r(\theta)u-\rho_u(t)u-f(u)\quad\text{on }(0,T]\times\R,\qquad u(0,\theta)=u_0(\theta).
\]

\medskip
\noindent\textbf{Step 6: Uniqueness and conclusion.}
By Proposition \ref{prop:uniqueness}, the solution in $X_T$ is unique; hence the whole sequence $(u_n)$ converges to $u$ on $[0,T]\times\R$ (not only a subsequence). Since $T>0$ was arbitrary, $u\in X_\infty$ is the unique global solution of \eqref{eq:main}. This completes the proof.
\end{proof}

\section{Persistence vs. extinction}\label{s:pers-vs-ext}

\subsection{Extinction for small data}\label{ss:proof-small-data}

\begin{proof}[Proof of Proposition \ref{prop:sigma-pt}] For $\sigma>0$, recall that we denote $u^\sigma$ the solution to \eqref{eq:main} starting from $u_0^\sigma$.  From the assumptions in \eqref{eq:family_initial_conditions},  and with $\varepsilon$ as in \eqref{eq:def_epsilon},  we may select $\sigma_*>0$ small enough so that, for any $0<\sigma<\sigma_*$, $\frac{e^{\rmax }}{\sqrt{4\pi}}\rho _{u^\sigma}^0<\varepsilon$. From now, we assume $0<\sigma<\sigma_*$. Since
$$
\partial _t u^\sigma \leq \partial _{\theta \theta}u^\sigma+\rmax u^\sigma,
$$
it follows from the standard parabolic comparison principle that
$$
u^\sigma (t,\theta)\leq e^{\rmax t} (G(t,\cdot)*u_0^\sigma)(\theta)\leq \frac{e^{\rmax t}}{\sqrt{4\pi t}}\Vert u_0^\sigma\Vert_{L^1}=\frac{e^{\rmax t}}{\sqrt{4\pi t}} \rho_{u^\sigma}^0,\quad \forall t\geq 0,\,  \forall \theta \in \R, 
$$
and, in particular, 
$$
u^\sigma(1,\theta)\leq \varepsilon, \quad \forall \theta \in \R.
$$

Now, consider the solution $U=U(t,\theta)$ to the Cauchy problem
\begin{equation*}
\begin{cases}
		 \partial_t U= \partial_{\theta\theta} U+ \rmax\, U- f(U), \quad  & t>0 ,\, \theta\in\R,
		\\
		U(0,\theta) = \varepsilon, & \theta \in \R.
	\end{cases}
\end{equation*}
From the standard parabolic comparison principle, there holds 
$u^\sigma(1+t,\theta)\leq U(t,\theta)$ for all $t\geq 0$, all $\theta \in \R$. Moreover, $U$ does not depend on   $\theta$, and, from $U(0)\leq \varepsilon$ and  since $f(0)=0$ and $\rmax s - f(s)<0$ for all $s\in (0,\varepsilon]$,  we have $U(t)\to 0$ as $t\to +\infty$. Hence $u^\sigma(t,\cdot) \to 0$ in $L^\infty(\R)$, as $t\to +\infty$, which concludes the proof. 
\end{proof}

\subsection{Extinction for large data}\label{ss:proof-large-data}

\begin{proof}[Proof of Theorem \ref{th:sigma-gd}]  Integrating equation~\eqref{eq:main} over $\theta \in \R$, we get
$$
   \rho_u'(t)  =  \int _{-\infty}^{+\infty} r(\theta)u(t,\theta)\, d\theta - \rho_u^2(t) - \int_{-\infty}^{\infty} f(u(t,\theta))\, d\theta.
   $$
Since $r$ is bounded from below, say by some $r_{min}$, and $f$ is $C_{Lip}$-Lipschitz continuous, we get
$$
 \rho_u'(t) \ge (r_{min}-C_{Lip}) \rho_u(t)  - \rho_u^2(t)\geq -K \rho_u(t)  - \rho_u^2(t),
$$
with $K:=C_{Lip}-\min(0,r_{min})>0$. From the comparison principle for ODEs, we have $\rho_u(t) \geq \rhod(t)$ for all $t\ge 0$, where $\rhod$ is the solution to the Cauchy problem
\begin{equation}
    \rhod'(t) = -K\,  \rhod(t)  - \rhod^2(t), \quad \rhod(0)=\rho_u(0)=:\rho_u^0,
\end{equation}
which is explicitly computable so that
\begin{equation}
   \rho_u(t)\geq  \rhod(t) =  \frac{K}{\left( 1+\frac{K}{\rho_u^0}\right) e^{K t}-1}, \quad \forall t\geq 0. 
\end{equation}

Now, consider the solution $\ub=\ub(t,\theta)$ to the Cauchy problem
\begin{equation}\label{eq:ub1}
\begin{cases}
		 \partial_t \ub = \partial_{\theta\theta} \ub + \rmax\, \ub - \rhod (t) \, \ub, \quad  & t>0 ,\, \theta\in\R,
		\\
		\ub(0,\theta) = M, & \theta \in \R,
	\end{cases}
\end{equation}
with $M$ coming from hypothesis  \eqref{eq:family_initial_conditions} on the family of initial conditions $(u_0^\sigma)_{\sigma \geq 0}$.  Then the solution $u^\sigma$ to \eqref{eq:main} with initial condition $u_0^\sigma$ is a subsolution of \eqref{eq:ub1}, and the standard parabolic comparison principle implies $u^\sigma(t,\theta)\leq \ub(t,\theta)$ for all $t\ge0$, all $\theta \in \R$. Moreover, $\ub$ does not depend on $\theta$, and is explicitly given by
$$
\ub(t)= M \exp\left(\rmax  t - \int_0^t \rhod(s)\,ds\right)
= \frac{M e^{\rmax   t}}{1 + \frac{\rho_{u^\sigma}^0}{K}(1-e^{-K  t})}.
$$

Coming back to the assumptions in \eqref{eq:family_initial_conditions} on the family $(u^\sigma_0)_{\sigma \ge 0}$, we have $\rho_{u^\sigma}^{0} =\|u_0^\sigma\|_{L^1} \to +\infty$ as $\sigma \to +\infty$.  Thus, we can define $\Sigma>0$ such that, for all $\sigma> \Sigma$, $\rho_{u^\sigma}^0>\rmax $.  We note that $\ub(0)=M$, $\ub(+\infty)=+\infty$ and, if $\sigma>\Sigma$, from the equation \eqref{eq:ub1},  $\ub$ reaches a unique minimum at $t^*=t^*_\sigma$ such that $\rhod(t^*)=\rmax $, where
$$
t^* = \frac{1}{K} \ln\left(\frac{\rho_{u^\sigma}^0 \, \rmax + \rho_{u^\sigma}^0\, K}{\rho_{u^\sigma}^0 \,\rmax  + \rmax \, K}\right),
$$
and
\begin{equation}\label{le-minimum}
    \ub(t^*)=M\, \left(\frac{\rmax +K}{\rho_{u^\sigma}^0 +K}\right)^{1+\frac{\rmax }K}\left(\frac{\rho_{u^\sigma}^0}{\rmax }\right)^{\frac{\rmax }K} \to 0, \quad \text{ as } \sigma \to +\infty.
\end{equation}
 Thus, we can define $\sigma^*>\Sigma$ such that, for all $\sigma> \sigma^*$, $\ub(t^*)<\varepsilon$,  with $\varepsilon$ defined as in~\eqref{eq:def_epsilon}.


Let us now define the solution $U=U(t,\theta)$ to the Cauchy problem
\begin{equation}\label{eq:ubb1}
\begin{cases}
	  \partial_t U = \partial_{\theta \theta} U + \rmax  \, U - f(U),	  \quad  & t>t^* ,\, \theta\in\R,
		\\
		U(t^*,\theta) = \ub(t^*), & \theta \in \R.
	\end{cases}
\end{equation}
From the above, for $\sigma>\sigma ^*$, $u^\sigma$ is a subsolution of~\eqref{eq:ubb1},  and the standard parabolic comparison principle implies $u^\sigma(t,\theta)\le U(t,\theta)$ for all $t\ge t^*$, $\theta \in \R$. Moreover, $U$ does not depend on   $\theta$, and, from $\ub(t^*)<\varepsilon$ and  since $f(0)=0$ and $\rmax s - f(s)<0$ for all $s\in (0,\varepsilon]$,  we have $U(t)\to 0$ as $t\to +\infty$. Hence $u^\sigma(t,\cdot) \to 0$ in $L^\infty(\R)$, as $t\to +\infty$, which concludes the proof. 
\end{proof}

\begin{rem} From the above proof, we may provide (at least when $r$ is bounded from below) a sufficient condition, relating the $L^\infty$ and $L^1$ norms  of a given initial datum $u_0$, for extinction. Indeed, in light of  \eqref{le-minimum}, it is sufficient to have
$$
\Vert u_0\Vert _{L^\infty} <\varepsilon g( \Vert u_0\Vert _{L^1}),\quad g(x):=\left(\frac{\rmax }{x}\right)^{\frac{\rmax }{K}}\left(\frac{x+K}{\rmax +K}\right)^{1+\frac{\rmax }{K}}, \quad \text{ and } \Vert u_0\Vert _{L^1}>\rmax .
$$
We note in particular that, for some $C=C(K,\rmax )>0$, $g(x) \sim  C x$ as $x\to +\infty$.
\end{rem}

\subsection{Survival may occur}\label{ss:proof-survival}

Before proving Theorem \ref{th:survival}, we start with some preparation. We denote $u=u(t,x)$ the solution to \eqref{eq:main} starting from some \lq\lq admissible'' $u_0$. If the fitness function is negative outside some interval, then the solution is expected to have an exponential decay as $\theta \to \pm \infty$. Precisely, the following holds.  

\begin{lemma} \label{lemma:u<M exp} Assume there are $R>0$ and $\rinf>0$ such that $r(\theta)\leq -\rinf$ for all $ \vert \theta\vert \geq R$. 

If there is $M>0$ such that
$$
u_0(\theta)\leq \bar{u}(\theta) := M \exp\left(-\sqrt{\rinf}(\vert \theta\vert -R)\right), \;\forall \vert \theta\vert  \geq R,\quad \text{ and } \quad 
u(t,R)\leq M, \; \forall t\geq 0,
$$
then 
$$
u(t,\theta)\leq \bar{u}(\theta) := M \exp\left(-\sqrt{\rinf}(\vert \theta\vert -R)\right), \;\forall t \geq 0, \forall \vert \theta\vert  \geq R.
$$
\end{lemma}

\begin{proof} It is enough to work on $(R,+\infty)$.
The function $\bar{u}=\bar{u}(\theta)$  is obviously the (constant in time) solution of the parabolic problem (with Dirichlet boundary conditions)
\begin{equation*}
\begin{cases}
	\dt v =\partial _{\theta\theta}v  -\rinf v,    \quad  & t>0 ,\, \theta\in (R,+\infty),
		\\
	v(t,R)=M,	 & t >0,\\
	v(0,\theta)=\bar u (\theta), & \theta \in (R,+\infty). 
	\end{cases}
\end{equation*}
From the assumptions, $u$ is a sub-solution of the above problem and the conclusion follows from the standard parabolic comparison principle. 
\end{proof}

Next, under some symmetry assumptions,  we have the following. 

\begin{lemma} \label{lemma:radially nonincreasing bound}
Assume both the fitness function $r$ and the initial datum $u_0$ are radially nonincreasing. Then, for all $t>0$, $u(t,\cdot)$ is radially nonincreasing and, for all $t\geq 0$, all $\theta\neq 0$, 
    \begin{equation} \label{eq:radially nonincreasing bound}
        u(t,\theta) \leq \frac{\rho(t)}{2\abs{\theta}},
    \end{equation}
    and
    \begin{equation} \label{eq:radially nonincreasing bound constant}
        u(t,\theta) \leq \frac{1}{2\abs{\theta}} \max(\rho_{u_0}, \rmax).
    \end{equation}
\end{lemma}

\begin{proof} Since $(t,\theta)\mapsto u(t,-\theta)$ solves the same Cauchy problem, we have, by uniqueness, that, for any $t>0$, $u(t,\cdot)$ is radial. By differentiating the equation \eqref{eq:main} and denoting $v(t,\theta):=\partial _\theta u (t,\theta)$, we have
$$
 \partial_t v=\partial_{\theta\theta} v+r'(\theta)u+r(\theta)v-v\rho_u(t)-vf'(u). 
$$
Since $u(t,\cdot)$ is radial we have $v(t,0)=0$, and since $r'\leq 0$ on $(0,+\infty)$ we have 
\begin{equation*}
\begin{cases}
	\partial_t v\leq \partial_{\theta\theta} v+v( r(\theta) -\rho_u(t)-f'(u)),    \quad  & t>0 ,\, \theta\in (0,+\infty),
		\\
	v(t,0)=0,	 & t >0,\\
	v(0,\theta)\leq 0, & \theta \in (0,+\infty). 
	\end{cases}
\end{equation*}
From the standard parabolic principle we deduce $v(t,\theta)\leq 0$ for all $t> 0$, $\theta \in (0,+\infty)$, and thus, for all $t>0$, $u(t,\cdot)$ is radially decreasing. Knowing this, we deduce that, for any $\theta\in \R$, $\rho(t)\geq \int_{-\vert \theta\vert}^{\vert \theta\vert} u(t,\theta')\,d\theta' \geq 2\vert \theta\vert u(t,\theta)$, that is \eqref{eq:radially nonincreasing bound} from which \eqref{eq:radially nonincreasing bound constant} immediately follows using the upper bound on $\rho$ of Proposition \ref{prop: u < M, rho < M'}.
\end{proof}

We are now in the position to complete the proof of Theorem \ref{th:survival}.

\begin{proof}[Proof of Theorem \ref{th:survival}] Let us consider a radially nonincreasing initial datum $u_0\in C^0_c(\R, [0,+\infty))$ with $0<\rho_{u_0}<\rmax$. This initial mass being fixed, we plan to construct an {\it ad hoc} example allowing survival. 

To do so we consider radially nonincreasing fitness functions $r=r_\alpha$ such that 
\begin{equation}\label{r-takes-the-form-alpha}
r(\theta)\begin{cases}=\rmax -\alpha^2 \theta ^2, \quad &\forall \theta \in \left(-\frac{\sqrt{2\rmax }}{\alpha},\frac{\sqrt{2\rmax }}{\alpha}\right),\\
\geq \rmax -\alpha^2 \theta ^2, &\forall \theta \notin \left(-\frac{\sqrt{2\rmax }}{\alpha},\frac{\sqrt{2\rmax }}{\alpha}\right),
\end{cases}
\end{equation}
for any $0< \alpha  \ll 1$ (which will be related to $\varepsilon$ measuring the Allee effect later). Without loss of generality we may assume
\begin{equation}\label{support}
\text{supp } u_0 \subset (-\sqrt{2\rmax}\alpha^{-1},\sqrt{2\rmax}\alpha^{-1}).
\end{equation}
In the sequel we denote by $c_i$ ($i=1,\dots$) various positive constants that depend only on $\rmax $ but we keep evident the crucial dependence on $\alpha$.

 From \eqref{eq:radially nonincreasing bound constant}, we  get $u(t,\frac{\sqrt{2\rmax}}{\alpha}) \leq c_1\alpha$. Since, for all $\theta \geq \frac{\sqrt{2\rmax}}{\alpha}$, $r(\theta) \leq r(\frac{\sqrt{2\rmax}}{\alpha}) = -\rmax$, we can use these two bounds and Lemma \ref{lemma:u<M exp}, thanks to \eqref{support}, to get
\begin{equation} \label{eq:u<Kr exp (proof persistence)}
    u(t,\theta) \leq c_1 \alpha \exp\left(- \sqrt  {\rmax} \left(\theta-\frac{c_2}{\alpha}\right)\right), \quad \forall t\geq 0, \forall \theta \geq  \frac{\sqrt{2\rmax}}{\alpha}=\frac{c_2}{\alpha}.
\end{equation}

Integrating equation~\eqref{eq:main} over $\theta \in \R$, recall that we get
\begin{equation}\label{edo-rho}
   \rho_u'(t)  =  \int _{-\infty}^{+\infty} r(\theta)u(t,\theta)\, d\theta - \rho_u^2(t) - \int_{-\infty}^{\infty} f(u(t,\theta))\, d\theta.
  \end{equation}
To estimate the second integral term in \eqref{edo-rho}, we write
\begin{align*}
    \nonumber
 \int_{-\infty}^{+\infty} f(u(t,\theta))\,d\theta=2   \int_0^{+\infty} f(u(t,\theta))\,d\theta
    &= 2\int_0^{\frac{c_2}\alpha} f(u(t,\theta))\,d\theta + 2\int_{\frac{c_2}{\alpha}}^{+\infty} f(u(t,\theta))\,d\theta,
    \\
    &\leq \frac{2c_2 \Vert f\Vert_{L^\infty}}{\alpha } + 2  \|f'\|_{L^\infty}  \int_{\frac{c_2}{\alpha}}^{+\infty} u(t,\theta)\, d\theta.
\end{align*}
 Since \eqref{eq:u<Kr exp (proof persistence)} provides
\begin{equation}\label{eq:integral u 3R/2 (proof persistence)}
\int_{\frac{c_2}{\alpha}}^{+\infty} u(t,\theta) \,d\theta\leq \int_{\frac{c_2}{\alpha}}^{+\infty} c_1 \alpha \exp\left(- \sqrt{\rmax} \left(\theta-\frac{c_2}{\alpha}\right)\right)\,d\theta=\frac{c_1}{ \sqrt{\rmax}}\alpha,
\end{equation}
we end up with
\begin{equation} \label{eq:int f bound (proof persistence)}
    \int_{-\infty}^{+\infty} f(u(t,\theta))\,d\theta
   \leq \frac{c_3 \Vert f\Vert_{L^\infty}}{\alpha}+ c_4  \|f'\|_{L^\infty}  \alpha.
\end{equation}
To estimate the first integral term in \eqref{edo-rho}, we write
$$
    \int_{-\infty}^{+\infty} r(\theta) u(t,\theta) \,d\theta=2\int_{0}^{+\infty} r(\theta) u(t,\theta) \,d\theta
    = 2\int_{0}^{\frac{c_2}{\alpha}} r(\theta) u(t,\theta) \, d\theta +2 \int_{\frac{c_2}{\alpha}}^{+\infty} r(\theta) u(t,\theta) \, d\theta.
$$
Since both $r$ and $u(t,\cdot)$ are  nonincreasing, we can use  the Chebyshev integral  inequality to get  (recall that $c_2=\sqrt{2r_{max}}$)
\begin{align}
    \nonumber
    \int_{0}^{\frac{c_2}{\alpha}} r(\theta) u(t,\theta) \,d\theta
    &\geq \frac{\alpha}{c_2} \left(\int_{0}^{\frac{c_2}{\alpha}} r(\theta) \, d\theta\right) \left(\int_{0}^{\frac{c_2}{\alpha}} u(t,\theta) \,d\theta\right)
    \\
    \nonumber
    &= \frac{\alpha}{c_2} 
    \left( \rmax\frac{c_2}{\alpha} - \frac{\alpha^2}{3}\left(\frac{c_2}{\alpha}\right)^3\right)
    \left(\frac{\rho_u(t)}{2}-\int_{\frac{c_2}{\alpha}}^{+\infty}u(t,\theta)d\theta\right)
    \\
    \nonumber
    &= \frac 13 \rmax
    \left(\frac{\rho_u(t)}{2}-\int_{\frac{c_2}{\alpha}}^{+\infty}u(t,\theta)d\theta\right)
    \\
    \label{eq:int 0 3R/2 ru (proof persistence)}
    &\geq \frac{\rmax}{6}\rho_u(t)-c_5\alpha,
\end{align}
where we used the bound \eqref{eq:integral u 3R/2 (proof persistence)}. Since \eqref{eq:u<Kr exp (proof persistence)} provides (note that $r\leq 0$ in $(\frac{c_2}{\alpha},+\infty)$)
\begin{align}
    \nonumber
    \int_{\frac{c_2}{\alpha}}^{+\infty} r(\theta) u(t,\theta) \, d\theta
    &\geq c_1\alpha \int_{\frac{c_2}{\alpha}}^{+\infty} \left(\rmax-\alpha^2\theta^2\right) 
    \exp\left(- \sqrt{\rmax} \left(\theta-\frac{c_2}{\alpha}\right)\right)\, d\theta
    \\
    \nonumber
    &=  c_1\alpha\left((r_{max}-c_2^2)I_0-2c_2\alpha I_1-\alpha^2 I_2 \right)\geq -c_6 \alpha
\end{align}
up to reducing $\alpha$ if necessary  (and where $I_k:=\int_0^{+\infty}z^ke^{-\sqrt{r_{max}}z}dz$). We end up with
\begin{equation}
\label{int-r-u-par-dessous}
\int_{-\infty}^{+\infty} r(\theta)u(t,\theta)\,d\theta\geq c_7\rho_u(t)-c_8 \alpha.
\end{equation}

Hence, inserting \eqref{eq:int f bound (proof persistence)} and \eqref{int-r-u-par-dessous} into \eqref{edo-rho}, we reach
\begin{equation*}
\rho _u '(t)\geq c_7 \rho_u(t)-\rho_u^2(t)-c_8\alpha-c_4  \|f'\|_{L^\infty} 
\alpha-\frac{c_3\Vert f\Vert _{L^\infty}}{\alpha},
\end{equation*}
where $f$ actually stands for $f_\varepsilon$ as in Assumption~\ref{ass:perturb} so that $\|f'\|_{L^\infty}=O(1)$  and $\|f\|_{L^\infty} = O(\varepsilon)$ as $\varepsilon \to 0$.  Because of that, we now choose the scaling $\alpha=\varepsilon ^\gamma$ with $0<\gamma<1$ (or $\alpha=A\varepsilon$ with $A>0$ sufficiently large for the following argument to hold, see Remark \ref{rem:gamma=1})  so that $c_8\alpha+c_4  \|f'\|_{L^\infty} \alpha+\frac{c_3\Vert f\Vert _{L^\infty}}{\alpha}=o(1)$ as $\varepsilon \to 0$. As a result, there is $\varepsilon_0>0$ such that, for any $0<\varepsilon<\varepsilon_0$, we have
\begin{equation*}
\rho _u '(t)\geq c_7 \rho_u(t)-\rho_u^2(t)-a, \quad \rho_u(0)=\rho_{u_0},
\end{equation*}
where $0<a\ll 1$ so that $\rho\mapsto c_7\rho-\rho^2-a$ has two positive real roots, the smallest one being denoted $\rho ^*$ and satisfying $0<\rho^*<\rho_{u_0}$. From the comparison principle for ODEs, we deduce that, for any $t\geq 0$, $\rho_u(t)\geq \rho^*>0$, which concludes the proof. 
\end{proof}

\subsection{Survival for large selection}\label{ss:survival-large-selection}

In this subsection, we let Assumption \ref{ass:large select} hold. We will note $u$ the solution to the problem (\ref{eq:main}) where the fitness function $r$ has the quadratic form $r(\theta)=\rmax-\alpha^2\theta^2$, with $1\leq \rmax \leq \alpha^2 \leq 2\rmax$.  The solution $u$ will  solely depend on $\rmax$, $\alpha$, the initial datum $u_0$, and the Allee effect $f$, the crucial scaling parameter being $r_{max}$ (see Assumption \ref{ass:large select}). We will also note $C_1$, $C_2$, $C_3$, ... positive constants which do not depend on any of the parameters.

We will show that, as we take a large $\rmax$ and after a short time elapses, we can make $u(t,\theta)$ as large as necessary in the range $\theta\in[-\frac{3\sqrt{\rmax}}{2\alpha}, \frac{3\sqrt{\rmax}}{2\alpha}]$ and as small as necessary outside of it.
From there, we can apply the same arguments presented in the proof of Theorem \ref{th:survival} to show that $u$ must persist.

We start with a simple bound of $u$ over a finite time interval which depends on $\alpha$.

\begin{lemma} \label{lemma:u<exp(tr)}
     For all $\theta\in\R$ and $0\leq t \leq \frac{1}{3\alpha}$, the solution $u$ satisfies:
     $u(t,\theta) \leq 2\rmax e^{t(\rmax - \frac{3\alpha^2}{4}\theta^2)}$.
\end{lemma}

\begin{proof}
    We let $\bar{u}(t,\theta) := 2\rmax e^{t(\rmax - \frac{3\alpha^2}{4}\theta^2)}$ for all $(t,\theta)\in[0,\frac{1}{3\alpha}]\times\R$, and we can verify that
    \begin{equation*}
        \dt \bar{u} - \dthth \bar{u} = \left[\rmax - \frac{3\alpha^2}{4}(1+3t^2\alpha^2)\theta^2 + \frac{3t\alpha^2}{2}\right] \bar{u},
    \end{equation*}
    over $(0,\frac{1}{3\alpha}]\times\R$.
    Since $0<t\leq \frac{1}{3\alpha}$, we have $1+3t^2\alpha^2 \leq \frac{4}{3}$, and thus $\dt \bar{u} - \dthth \bar{u} \geq (\rmax-\alpha^2\theta^2)\bar{u}$.
    Note as well that $\bar{u}(0,\theta) = 2\rmax \geq u_0(\theta)$ for all $\theta\in\R$.
    Applying the comparison principle, we get $u(t,\theta)\leq \bar{u}(t,\theta)$ for all $(t,\theta)\in[0,\frac{1}{3\alpha}]\times\R$.
\end{proof}

Next, we show that, after a short time, the integrals of $u(t,\theta)$ and $r(\theta)u(t,\theta)$ outside the range $[-\frac{3\sqrt{\rmax}}{2\alpha}, \frac{3\sqrt{\rmax}}{2\alpha}]$ go vanishingly small as $\rmax$ grows larger.

\begin{prop} \label{prop:int u and int ru}
	The solution $u$ verifies:
    \begin{equation*}
        \lim_{\rmax\to+\infty} \int_{|\theta|\geq\frac{3\sqrt{\rmax}}{2\alpha}} u(t,\theta) d\theta
        = \lim_{\rmax\to+\infty} \int_{|\theta|\geq\frac{3\sqrt{\rmax}}{2\alpha}} r(\theta)u(t,\theta) d\theta
        = 0,
    \end{equation*}
    uniformly w.r.t. all parameters and $t\geq \frac{1}{3\alpha}$.
\end{prop}

\begin{proof}
    Let us first bound $\rho_u$ from above.
    Since we assume $\alpha\geq 1$, we have $\frac{1}{3\alpha^2} \leq \frac{1}{3\alpha}$.
    When $0< t \leq \frac{1}{3\alpha^2}$, we can integrate the inequality of Lemma \ref{lemma:u<exp(tr)} with respect to $\theta$, and get $\rho_u(t) \leq 2\rmax e^{\frac{\rmax}{3\alpha^2}} \sqrt{\frac{4\pi}{3t\alpha^2}} \leq C_1 \frac{\rmax}{\alpha\sqrt{3t}}$ by using the inequality $\frac{\rmax}{\alpha^2} \leq 1$ from the assumptions of this subsection. We may suppose  $C_1>1$.
    For $t\geq\frac{1}{3\alpha^2}$, we apply Proposition \ref{prop: u < M, rho < M'} (with $\frac{1}{3\alpha^2}$ as the initial time and $M'_0=C_1\rmax$), giving us $\rho_u(t) \leq \max(C_1\rmax,\rmax)=C_1\rmax$ for all $t\geq\frac{1}{3\alpha^2}$.

    We then prove a global bound on $u$ by applying arguments found in the proof of Proposition \ref{prop: u < M, rho < M'}.
    First, define $\bar{u}(t,\theta):=2\rmax e^{\rmax t}$ the solution to the equation $\dt \bar{u} - \dthth \bar{u} = \rmax\bar{u}$, which verifies $\bar{u}(0,\theta)=2\rmax \geq u_0(\theta)$.
    $u$ is a subsolution of this equation, and so $u\leq \bar{u}$ according to the comparison principle.
    When $0\leq t \leq \frac{2}{3\alpha^2}$, we thus have $u(t,\theta)\leq 2\rmax e^{\frac{2\rmax}{3\alpha^2}}\leq C_2\rmax$ for all $\theta\in\R$, using the assumption $\rmax \leq \alpha^2$.
    For $t\geq \frac{2}{3\alpha^2}$, we apply Lemma \ref{lemma: u < C rho} (with $\tau=\frac{1}{3\alpha^2}$) to get
    $u(t+\frac{1}{3\alpha^2},\theta)
	\leq e^{\frac{\rmax}{3\alpha^2}} \sqrt{\frac{3\alpha^2}{4\pi}} \rho_u(t) \leq C_3 \rmax^{3/2}$
	for all $t \geq \frac{1}{3\alpha^2}$, thanks to our bound on $\rho_u$ and the assumption $\rmax\leq\alpha^2\leq2\rmax$.
	Putting our two estimates of $u$ together, we have:
	\begin{equation}	 \label{eq:u<15r}
		u(t,\theta) \leq \max(C_2\rmax,C_3\rmax^{3/2}) \leq C_4\rmax^{3/2},
	\end{equation}
	for all $(t,\theta) \in [0,+\infty)\times\R$, since we suppose $\rmax\geq 1$.

    Now, we consider $t= \frac{1}{3\alpha}$ and $\vert \theta \vert \geq \sqrt{\frac{4r_{max}}{ 3 \alpha^2}}$. For those $\theta$,  Lemma \ref{lemma:u<exp(tr)} provides $u(\frac{1}{3\alpha},\theta) \leq 2\rmax \exp(\frac{1}{3\alpha}(\rmax - \frac{3\alpha^2}{4}\theta^2))$ so that
    \begin{equation*}
		u\left(\frac{1}{3\alpha},\theta\right) \leq 2\rmax \exp\left(-\sqrt{\frac{\rmax}{3}}\left(|\theta| - \sqrt{\frac{4\rmax}{3\alpha^2}}\right)\right),
	\end{equation*}
	after straightforward computations using the concavity of $\theta\mapsto \frac{1}{3\alpha}\left( r_{max}-\frac{3\alpha^2
    }{4}\theta^2\right)$. Up to increasing $r_{max}$ if necessary, we deduce that 
     \begin{equation*}
		u\left(\frac{1}{3\alpha},\theta\right) \leq C_4 \rmax^{3/2} \exp\left(-\sqrt{\frac{\rmax}{3}}\left(|\theta| - \sqrt{\frac{4\rmax}{3\alpha^2}}\right)\right).
	\end{equation*}
    
    Hence, recalling \eqref{eq:u<15r},
    an application of Lemma \ref{lemma:u<M exp} (with $t=\frac{1}{3\alpha}$ as \lq\lq initial time", $R=\frac{5\sqrt{\rmax}}{4\alpha}$, $\underline{r}=\frac{r_{max}}{3}$, $M=C_4r_{max}^{3/2}$) yields
	\begin{equation*}
		u(t,\theta) \leq C_4\rmax^{3/2} \exp\left(-\sqrt{\frac{\rmax}{3}}\left(|\theta| - \frac{5\sqrt{\rmax}}{4\alpha}\right)\right), \quad \forall t\geq \frac{1}{3\alpha},\quad \forall \vert \theta \vert \geq \frac{5\sqrt{\rmax}}{4\alpha}.
	\end{equation*}
As a result, for any $t\geq \frac{1}{3\alpha}$, 
	\begin{align*}
		\int_{|\theta|\geq\frac{3\sqrt{\rmax}}{2\alpha}} u(t,\theta) d\theta
		&\leq
		2C_4\rmax^{3/2}
		\int_{\frac{3\sqrt{\rmax}}{2\alpha}}^{+\infty}
		\exp\left(-\sqrt{\frac{\rmax}{3}}\left(\theta - \frac{5\sqrt{\rmax}}{4\alpha}\right)\right) d\theta
		\\
        &\leq C_5\rmax \exp\left(-\frac{\rmax}{4\sqrt{3}\alpha}\right)
		\leq C_{6} e^{-C_{7}\sqrt{\rmax}},
	\end{align*}
    since $1\leq r_{max}\leq \alpha^2\leq 2r_{max}$. This proves the first result of the proposition. Similarly, for any $t\geq \frac{1}{3\alpha}$, 
	\begin{align*}
		\int_{|\theta|\geq\frac{3\sqrt{\rmax}}{2\alpha}} |r(\theta) u(t,\theta)| d\theta
		&\leq
		2C_4\rmax^{3/2} \int_{\frac{3\sqrt{\rmax}}{2\alpha}}^{+\infty} (\rmax+\alpha^2\theta^2)
		\exp\left(-\sqrt{\frac{\rmax}{3}}\left(\theta - \frac{5\sqrt{\rmax}}{4\alpha}\right)\right) d\theta
		\\
		&=
		2C_4 e^{-\frac{\rmax}{4\sqrt{3}\alpha}} (6\sqrt{3}\alpha^2 + 9\rmax\alpha + \frac{13}{4}\sqrt{3} \rmax^2)
        \leq C_{8} e^{-C_{9}\sqrt{\rmax}},
	\end{align*}
%
   %
    %
    using the assumptions $1\leq\rmax\leq\alpha^2\leq 2\rmax$ and the boundedness of $xe^{-\sqrt{x}}$, $x^{3/2}e^{-\sqrt{x}}$, and $x^2 e^{-\sqrt{x}}$ for $x\geq 0$.
    Hence, we get the second result of the proposition.
\end{proof}

We now prove that, at the small time $t=\frac{1}{3\alpha}$, $\rho$ grows arbitrarily large as $\rmax$ grows larger.

\begin{prop} \label{prop:rho>r} The mass $\rho_u=\rho_u(t)$ of the solution $u$ verifies 
    \begin{equation*}
        \lim_{\rmax\to+\infty} \rho_u\left(\frac{1}{3\alpha}\right) = +\infty,
    \end{equation*}
    uniformly w.r.t. all parameters.
\end{prop}

\begin{proof}
    As shown in the proof of Proposition \ref{prop:int u and int ru}, we have $\rho_u(t) \leq C_1\frac{\rmax}{\alpha\sqrt{3t}}$ for all $t\in(0,\frac{1}{3\alpha^2}]$ with $C_1>1$. 
    Integrating the main equation (\ref{eq:main}) with respect to $\theta$, we get the inequality $\rho_u'(t) \leq \rmax \rho_u(t) - \rho_u(t)^2$ for all $t>0$.
    We define the function $\bar{\rho}(t) := \frac{\rmax}{1-(1-\frac{1}{C_1})\exp(-\rmax(t-\frac{1}{3\alpha^2}))}$, which verifies $\bar{\rho}(\frac{1}{3\alpha^2})=C_1\rmax \geq \rho_u(\frac{1}{3\alpha^2})$, and $\bar{\rho}'(t) = \rmax \bar{\rho}(t) - \bar{\rho}(t)^2$ for all $t>\frac{1}{3\alpha^2}$, and thus $\rho_u(t) \leq \bar{\rho}(t)$ for all $t>\frac{1}{3\alpha^2}$ by the comparison principle.
    We can write, for $t>\frac{1}{3\alpha^2}$:
    \begin{align*}
        \bar{\rho}(t)
        &= \rmax - \rmax\left(1 - \frac{1}{1-(1-\frac{1}{C_1})\exp(-\rmax(t-\frac{1}{3\alpha^2}))}\right)
        \\
        &= \rmax + \rmax \frac{(1-\frac{1}{C_1})\exp(-\rmax(t-\frac{1}{3\alpha^2}))}{1-(1-\frac{1}{C_1})\exp(-\rmax(t-\frac{1}{3\alpha^2}))}
        \\
        &\leq \rmax + \rmax \frac{(1-\frac{1}{C_1})\exp(-\rmax(t-\frac{1}{3\alpha^2}))}{1-(1-\frac{1}{C_1})}
        = \rmax + (C_1-1)\rmax e^{-\rmax(t-\frac{1}{3\alpha^2})}.
    \end{align*}
    Thus, defining the following continuous function
    \begin{equation*}
        \rho_1(t):= 
        \begin{cases}
            C_1\frac{\rmax}{\alpha\sqrt{3t}}
            &\text{ if } 0< t \leq \frac{1}{3\alpha^2},
            \\
            \rmax + (C_1-1)\rmax e^{-\rmax(t-\frac{1}{3\alpha^2})}
            &\text{ if } t > \frac{1}{3\alpha^2},
        \end{cases}
    \end{equation*}
    we have $\rho_1(t) \geq  \rho_u(t)$ for all $t>0$.
    We also note that $\rho_1 \geq \rmax$ since $C_1>1$.

    We now estimate the term $\int_0^t \rho_1(s)ds$. We may assume $t\geq \frac{1}{3\alpha^2}$. Obviously $\int_0^{\frac{1}{3\alpha^2}} \rho_1(s)ds = C_1\frac{2\rmax}{3\alpha^2} \leq C_1$,
    and
    $\int_{\frac{1}{3\alpha^2}}^t \rho_1(s)ds
    = \rmax(t-\frac{1}{3\alpha^2}) + (C_1-1)(1 - e^{-\rmax(t-\frac{1}{3\alpha^2})})
    \leq \rmax t + C_1-1$
    so that
    \begin{equation} \label{eq:int rho1<rt+18}
        \int_0^t \rho_1(s)ds
        = \int_0^{\frac{1}{3\alpha^2}} \rho_1(s)ds + \int_{\frac{1}{3\alpha^2}}^t \rho_1(s)ds
        \leq \rmax t + C_{10}.
    \end{equation}

   We can now prove the bound from below we are seeking.
    Define, for all $(t,\theta) \in (0,+\infty)\times\R$:
    \begin{equation*}
        \tilde{u}(t,\theta) := \rmax \exp\left(\rmax t - \int_0^t\rho_1(s)ds - \alpha(t+\frac{\theta^2}{2})\right).
    \end{equation*}
    Obviously
    $\dt\tilde{u}(t,\theta) = (\rmax-\rho_1(t) - \alpha) \tilde{u}(t,\theta)$
    and
    $\dthth\tilde{u}(t,\theta) = (-\alpha + \alpha^2\theta^2)\tilde{u}(t,\theta)$,
    and thus
    $\dt\tilde{u}(t,\theta) - \dthth\tilde{u}(t,\theta)
    = \left( \rmax - \alpha^2\theta^2 - \rho_1(t)\right)\tilde{u}(t,\theta).$
    Since $\rho_1 \geq \rmax$, we have $\rmax-\alpha^2\theta^2-\rho_1(t) \leq 0$ for all $(t,\theta)\in(0,+\infty)\times\R$.
    Therefore, defining the function $\underline{u}(t,\theta) := \tilde{u}(t,\theta) - \rmax e^{-\frac{\alpha}{8}} - 2\rmax t$, we get the inequality:
    \begin{align*}
        \dt\underline{u}(t,\theta) - \dthth\underline{u}(t,\theta)
        &= \left( \rmax - \alpha^2\theta^2 - \rho_1(t)\right)\left(\underline{u}(t,\theta) + \rmax e^{-\frac{\alpha}{8}} + 2\rmax t\right) - 2\rmax
        \\
        &\leq \left( \rmax - \alpha^2\theta^2 - \rho_1(t)\right)\underline{u}(t,\theta) - 2\rmax.
    \end{align*}
    Note that $\underline{u}(t,\frac{1}{2})=\underline{u}(t,-\frac{1}{2}) \leq 0$ for all $t>0$, and $\underline{u}(0,\theta) \leq \rmax$.
    Moreover, since $\norm{f}_{L^\infty([0,+\infty))}\leq 2\rmax$ (see Assumption \ref{ass:large select}), one has $r(\theta)u - \rho_u(t)u - f(u) \geq (\rmax - \alpha^2\theta^2 - \rho_1(t))u - 2\rmax$ so that 
    \begin{equation*}
   \begin{cases}
            \dt u - \dthth u
            \geq (\rmax - \alpha^2\theta^2 - \rho_1(t))u - 2\rmax, &  (t,\theta) \in (0,+\infty)\times(-\frac 12,\frac 12),
            \\
            u(t,\pm 1/2) \geq 0,
            & t>0,
            \\
            u(0,\theta) = u_0(\theta) \geq \rmax, & \theta\in[-\frac 12,\frac 12].
        \end{cases}
    \end{equation*}
    From the comparison principle, we thus get $u(t,\theta) \geq \underline{u}(t,\theta)$ for all $(t,\theta)\in[0,+\infty)\times[-1/2,1/2]$.
    Integrating with respect to $\theta$, we have the inequality $\rho_u(t) \geq \int_{-1/2}^{1/2} u(t,\theta)d\theta \geq \int_{-1/2}^{1/2} \underline{u}(t,\theta)d\theta$.
    Note that
    \begin{align*}
        \int_{-1/2}^{1/2} e^{-\frac{\alpha}{2}\theta^2} d\theta
        &= \int_{-\infty}^{+\infty} e^{-\frac{\alpha}{2}\theta^2} d\theta - 2\int_{1/2}^{+\infty} e^{-\frac{\alpha}{2}\theta^2} d\theta
        \\
        &\geq \sqrt{\frac{2\pi}{\alpha}} - 2\int_{1/2}^{+\infty} e^{-\frac{\alpha}{2}(\theta-\frac{1}{4})} d\theta
        = \sqrt{\frac{2\pi}{\alpha}} - \frac{4}{\alpha} e^{-\frac{\alpha}{8}},
    \end{align*}
    where we used the fact that $\theta^2 \geq \theta - \frac{1}{4}$ for all $\theta\in\R$.
    Therefore, from the definition of $\tilde{u}$, we have
    \begin{align*}
        \int_{-1/2}^{1/2} \underline{u}(t,\theta)d\theta
        &\geq \rmax e^{\rmax t - \int_0^t\rho_1(s)ds - t\alpha} \left( \sqrt{\frac{2\pi}{\alpha}} - \frac{4}{\alpha} e^{-\frac{\alpha}{8}} \right) 
        - \rmax e^{-\frac{\alpha}{8}}
        - 2\rmax t.
    \end{align*}
    Using the estimate (\ref{eq:int rho1<rt+18}) of $\rho_1$ we found above, we have, at $t=\frac{1}{3\alpha} \geq \frac{1}{3\alpha^2}$:
    \begin{align*}
        \int_{-1/2}^{1/2} \underline{u}\left(\frac{1}{3\alpha},\theta\right)d\theta
        &\geq
        \rmax e^{-C_{10} - \frac{1}{3}} \sqrt{\frac{2\pi}{\alpha}}
        - \frac{4\rmax}{\alpha} e^{-\frac{\alpha}{8}} 
        - \rmax e^{-\frac{\alpha}{8}}
        - \frac{2\rmax}{3\alpha}
        \\
        &\geq 
         \sqrt   {\sqrt 2 \pi} e^{-C_{10}-\frac{1}{3}} \rmax^{3/4}
        - 4\sqrt{\rmax}
        - \rmax e^{-\frac{\sqrt{\rmax}}{8}}
        - \sqrt{\rmax}
        \\
        &\geq
        C_{11} \rmax^{3/4}
        - C_{12}\sqrt{\rmax}
    \end{align*}
    where, to simplify the estimate, we used the assumption $1\leq\rmax\leq\alpha^2\leq 2r_{max}$ and the inequality $x e^{-x} \leq 1$ for all $x\geq 0$.
    We thus have $\rho_u(\frac{1}{3\alpha}) \geq C_{11} \rmax^{3/4} - C_{12} \sqrt{\rmax}$, giving us the desired limit as $\rmax\to+\infty$.
\end{proof}

We can now move on to the proof of Theorem \ref{th:survival large alpha}.

\begin{proof}[Proof of Theorem \ref{th:survival large alpha}]
    Let Assumption \ref{ass:large select} hold.
    In the sequel, we will only consider $\sigma\geq 1$. Thanks to Proposition \ref{prop:int u and int ru}, we also suppose $\rmax$ is large enough so that $\int_{|\theta|\geq\frac{3\sqrt{\rmax}}{2\alpha}} u(t,\theta) d\theta \leq 1$ and $\abs{\int_{|\theta|\geq\frac{3\sqrt{\rmax}}{2\alpha}} r(\theta)u(t,\theta) d\theta} \leq 1$ for all $t\geq \frac{1}{3\alpha}$.
    We note that, as seen in Lemma \ref{lemma:radially nonincreasing bound}, the fact that the initial condition $u_0^\sigma$ of Assumption \ref{ass:large select} is radially nonincreasing implies the corresponding solution $u^\sigma$ is itself radially nonincreasing.
    Integrating the main equation (\ref{eq:main}) with respect to $\theta$, we get
    \begin{equation}\label{truc}
        \rho_{u^\sigma}'(t) = \int_{-\infty}^{+\infty} r(\theta) u^\sigma(t,\theta) d\theta - \rho_{u^\sigma}(t)^2 - \int_{-\infty}^{+\infty} f(u^\sigma(t,\theta))d\theta.
    \end{equation}
    Thanks to the above inequalities, we can write $\int_{-\infty}^{+\infty} r(\theta) u^\sigma(t,\theta) d\theta
    \geq \int_{-\frac{3\sqrt{\rmax}}{2\alpha}}^{\frac{3\sqrt{\rmax}}{2\alpha}} r(\theta)u(t,\theta)d\theta - 1$ for all $t\geq\frac{1}{3\alpha}$.
    Since $r$ and $u^\sigma(t,\cdot)$ are both  radially nonincreasing, we can apply Chebyshev's integral inequality to get:
    \begin{align*}
        \int_{-\frac{3\sqrt{\rmax}}{2\alpha}}^{\frac{3\sqrt{\rmax}}{2\alpha}} r(\theta)u^\sigma(t,\theta)d\theta
        &\geq
        \frac{\alpha}{3\sqrt{\rmax}}
        \left(\int_{-\frac{3\sqrt{\rmax}}{2\alpha}}^{\frac{3\sqrt{\rmax}}{2\alpha}} \left(\rmax - \alpha^2\theta^2\right) d\theta\right)
        \left(\int_{-\frac{3\sqrt{\rmax}}{2\alpha}}^{\frac{3\sqrt{\rmax}}{2\alpha}} u^\sigma(t,\theta)d\theta\right)
        \\
        &=
        \frac{\rmax}{4} \left(\rho_{u^\sigma}(t) - \int_{|\theta|\geq\frac{3\sqrt{\rmax}}{2\alpha}} u^\sigma(t,\theta)d\theta\right)
        \geq
        \frac{ \rmax}{4} (\rho_{u^\sigma}(t) - 1),
    \end{align*}
    which yields $\int_{-\infty}^{+\infty} r(\theta) u^\sigma(t,\theta) d\theta \geq \frac{\rmax}{4}(\rho_{u^\sigma}(t) - 1) - 1$ for all $t\geq \frac{1}{3\alpha}$.
    For the last term in \eqref{truc}, we use Assumption \ref{ass:large select} to write
    \begin{align*}
        \int_{-\infty}^{+\infty} f(u^\sigma(t,\theta))d\theta
        &=
        \int_{-\frac{3\sqrt{\rmax}}{2\alpha}}^{\frac{3\sqrt{\rmax}}{2\alpha}} f(u^\sigma(t,\theta))d\theta
        + \int_{|\theta|\geq\frac{3\sqrt{\rmax}}{2\alpha}} f(u^\sigma(t,\theta))d\theta
        \\
        &\leq
        \frac{3\sqrt{\rmax}}{\alpha} \norm{f}_{L^\infty([0,+\infty))}
        + \norm{f'}_{L^\infty([0,+\infty))}\int_{|\theta|\geq\frac{3\sqrt{\rmax}}{2\alpha}} u^\sigma(t,\theta) d\theta
        \\
        &\leq
        6\rmax\frac{\sqrt{\rmax}}{\alpha}
        + 2\rmax
        \leq 8\rmax,
    \end{align*}
    for all $t\geq\frac{1}{3\alpha}$.
    Putting all our estimates together into \eqref{truc}, we have
    \begin{align*}
        \rho_{u^\sigma}'(t) \geq \frac{\rmax}{4}(\rho_{u^\sigma}(t) - 33) - \rho_{u^\sigma}(t)^2 - 1,
    \end{align*}
    for all $t\geq \frac{1}{3\alpha}$.
    We see that, if $\rmax$ is large enough, the polynomial $\rho\mapsto \frac{\rmax}{4}(\rho-33)-\rho^2-1$ has two positive real roots, the lowest of which we note $\rho^*$.
    We note as well that $\rho^*$ converges to $33$ as $\rmax\longrightarrow+\infty$, and we thus suppose $\rmax$ to be large enough so that $\rho^* < 34$.
    Using Proposition \ref{prop:rho>r}, we also suppose $\rmax$ is large enough so that $\rho_{u^\sigma}(\frac{1}{3\alpha}) > 34 > \rho^*$.
    We treat $t=\frac{1}{3\alpha}$ as an initial time and apply the comparison principle to the system above, which allows us to bound $\rho_{u^\sigma}$ from below by a subsolution that converges to the largest positive root of the polynomial in $\rho$.
    Therefore, $\liminf_{t\to+\infty} \rho_{u^\sigma}(t) > 0$ and $u^\sigma$ persists.
\end{proof}

\subsection{Existence of two stationary states}\label{ss:proof-two-stationary}

\begin{proof}[Proof of Theorem \ref{th:etats-stats}]
Recall that the fitness function is here assumed constant $r(\theta)=r$. Let $\lambda \in [0,r)$ and set $$g_\lambda(s):= (\lambda-r)s +f(s) \hbox{ and }G_\lambda(v) := \int_0^v g_\lambda(s)\, ds.$$We show that there exists a unique $\alpha_\lambda>0$  such that 
\begin{equation}\label{eq:alphalambda*}
    G_\lambda (\alpha_\lambda)= 0.
\end{equation}
Assumption~\ref{ass:bump} on $f$ imply that there is a unique $s_\lambda \in [\varepsilon,2\varepsilon)$ such that $g_\lambda(s) >0$ for $s\in (0,s_\lambda)$ and $g_\lambda(s) <0$ for $s> s_\lambda$. Thus the function $G_\lambda$ is increasing in $(0, s_\lambda)$ and decreasing in $(s_\lambda, 2\, \varepsilon)$. Moreover, we have $G_\lambda(0) = 0$ and, from the left inequality in \eqref{eq:hyp_tech_f}, $G_\lambda(2\,\varepsilon) > 0$.
Additionally, $G_\lambda(v) = G_\lambda(2 \,\varepsilon) + \frac{\lambda -r }{2}\, (v^2 - 4\, \varepsilon^2) \to -\infty$ as $v \to +\infty$. Finally, this shows that $\alpha_\lambda$ is uniquely defined by \eqref{eq:alphalambda*}, and that $\alpha_\lambda > 2 \, \varepsilon$. Moreover, 
\begin{equation} \label{eq:alphalambda_expr}
    (r-\lambda)\frac{\alpha_\lambda^2}{2}=\int_0^{2\varepsilon} f(s)\, ds,
\end{equation}
thus $\lambda\mapsto\alpha_\lambda$ is increasing in $[0,r)$ and 
\begin{equation}\label{eq:alphalambdainfty}
\alpha_\lambda \to +\infty\hbox{ as }\lambda \to r.    
\end{equation}

The Cauchy-Lipschitz theorem provides the existence and uniqueness of the maximal solution $p_\lambda$ to the ODE Cauchy problem
\begin{equation}\label{eq:ode1}
	\left\{
	\begin{aligned}
		& p_\lambda''(\theta) = g_\lambda(p_\lambda), \ \theta>0
		\\
		&p_\lambda(0) = \alpha_\lambda \hbox{ and }p_\lambda'(0)=0.
	\end{aligned}
	\right.
\end{equation}
Moreover, since $\alpha_\lambda > 2 \, \varepsilon$, by continuity, we can define 
$$\theta_0:=\sup\{\theta \hbox{ s.t. }p_\lambda> 2 \,\varepsilon \hbox{ in }(0,\theta)\}>0.$$
In the interval $\theta\in [0,\theta_0]$, $f(p_\lambda(\theta))=0$ and
\begin{equation} \label{eq:sol_cos}
    p_\lambda(\theta)= \alpha_\lambda \cos(\theta \sqrt{r-\lambda}), \hbox{ for all }\theta \in [0,\theta_0].
\end{equation}
Thus, we have 
\begin{equation}\label{eq:theta0}
    \theta_0=\frac{\arccos(2\varepsilon/\alpha_\lambda)}{\sqrt{r-\lambda}},
\end{equation}
and
\begin{equation} \label{eq:int1}
    \int_0^{\theta_0} p_\lambda (\theta)\, d\theta= \sqrt{\frac{\alpha_\lambda^2-4\varepsilon^2}{r-\lambda}}.
\end{equation}

The solution $p_\lambda$ of \eqref{eq:ode1} satisfies $p_\lambda''(\theta)=g_\lambda(p_\lambda(\theta))$. Multiplying this equation by $p_\lambda'$ and integrating over $(0,\theta)$ we get
\begin{equation}\label{eq:EDO1deg}
    (p_\lambda')^2(\theta)-(p_\lambda')^2(0) = 2\, \int_0^\theta p_\lambda'(z) g_\lambda(p_\lambda(z)) \, dz.
\end{equation}
We already know that $p_\lambda$ is decreasing in $(0,\theta_0]$, with $p'(\theta_0)<0$ (from \eqref{eq:sol_cos} and the definition of $\theta_0$). We now define
$$\theta_1:=\sup\{\theta \ge \theta_0 \hbox{ s.t. }p_\lambda'< 0 \hbox{ in }[\theta_0,\theta)\}.$$By continuity of $p'_\lambda,$ we have $\theta_1>\theta_0$. 

Using \eqref{eq:EDO1deg} together with $p_\lambda'(0)=0$, for $\theta\in (0,\theta_1]$:
\begin{align*}
    (p_\lambda')^2(\theta) & = 2 G_\lambda(p_\lambda(\theta))-2G_\lambda(p_\lambda(0))\\ & =  2 G_\lambda(p_\lambda(\theta))-2 G_\lambda(\alpha_\lambda) \\ & =  2 G_\lambda(p_\lambda(\theta)),
\end{align*}
from the definition~\eqref{eq:alphalambda*} of $\alpha_\lambda$.  Thus, $G_\lambda(p_\lambda(\theta))>0$ in $(0,\theta_1)$ and we have
\begin{equation}\label{eq:ODE2}
	\left\{
	\begin{aligned}
		&  p_\lambda'(\theta) = -\sqrt{2 G_\lambda(p_\lambda(\theta))}, \ \theta \in (\theta_0,\theta_1], \\
		& p_\lambda(\theta_0) = 2 \, \varepsilon.
	\end{aligned}
	\right.
\end{equation}
Assume that $\theta_1$ is finite. By continuity of $p'_\lambda,$ we get $p_\lambda'(\theta_1)=0$  and therefore $G_\lambda(p_\lambda(\theta_1)) = 0$. We recall that the function $G_\lambda$ is positive in $(0,2\varepsilon]$ and $G_\lambda(0)=0$. Thus necessarily $p_\lambda(\theta_1)=0$. Moreover, the regularity assumptions on $f$ imply that $g_\lambda$ is globally $K-$Lipschitz continuous, for some $K>0$. Thus,
$\ds G_\lambda(v)=\int_0^v g_\lambda(s)\, ds \le  K \frac{v^2}{2},$ which implies that $\sqrt{G_\lambda}$ is locally Lipschitz continuous at $0$. Cauchy-Lipschitz theorem then contradicts $p_\lambda(\theta_1)=0$ (thus, $\theta_1=+\infty$) and implies  that $p_\lambda>0$ in $(\theta_0,+\infty)$.

Let us find an upper bound for $p_\lambda$. We note that
\begin{equation*}
    G_\lambda(s)= \frac{\lambda \, s^2}{2} + G_0(s).
\end{equation*}
As already noted at the beginning of the proof, $G_0>0$ in $(0,2\, \varepsilon)$, thus, in this interval,
\begin{equation} \label{eq:ineqG}
    G_\lambda(s) >  \frac{\lambda \, s^2}{2}.
\end{equation}
Since $p_\lambda'(\theta)<0$ in $(0,+\infty)$, $p_\lambda(\theta)<2 \, \varepsilon$ for all $\theta \in (\theta_0,+\infty)$, and using~\eqref{eq:ODE2} together with~\eqref{eq:ineqG}, we get:
\begin{equation}\label{eq:ODE3}
	\left\{
	\begin{aligned}
		&  p_\lambda'(\theta) \le -\sqrt{\lambda} \, p_\lambda(\theta), \ \theta \in (\theta_0,+\infty), \\
		& p_\lambda(\theta_0) = 2 \, \varepsilon.
	\end{aligned}
	\right.
\end{equation}
Therefore, 
\begin{equation}\label{eq:sur-solp}
    p_\lambda(\theta) \le 2 \, \varepsilon \, e^{ -\sqrt{\lambda} (\theta-\theta_0)} \hbox{ for }\theta>\theta_0.
\end{equation}
Finally, using \eqref{eq:int1} and~\eqref{eq:sur-solp}, we get
\begin{equation}\label{eq:intp/2}
   \sqrt{\frac{\alpha_\lambda^2-4\varepsilon^2}{r-\lambda}} \le \int_0^{+\infty} p_\lambda(\theta)\, d\theta  \le \sqrt{\frac{\alpha_\lambda^2-4\varepsilon^2}{r-\lambda}} + \frac{2\,\varepsilon}{\sqrt{\lambda}}.
\end{equation}

Let us extend the function $p_\lambda$ to $\mathbb{R}$ by setting $p_\lambda(\theta) = p_\lambda(-\theta)$ for all $\theta < 0$. We readily check that $p_\lambda \in C^2(\mathbb{R})$. Moreover, setting 
$$j: \ (0,r) \to \R, \ \lambda \mapsto \int_{-\infty}^{+\infty} p_\lambda(\theta) \, d\theta - \lambda,$$we get that
$p_\lambda$ is a nontrivial stationary solution of our main problem~\eqref{eq:main} if and only if $j(\lambda)=0$. 
By the continuous dependence of $p_\lambda$ on the parameter $\lambda$ (via the Cauchy-Lipschitz theorem), together with the exponential decay \eqref{eq:sur-solp}, the continuity of $j$ in $(0,r)$ follows by dominated convergence. 
Using \eqref{eq:intp/2} together with the properties $\alpha_\lambda> 2\varepsilon$ for all $\lambda \in [0,r)$ and~\eqref{eq:alphalambdainfty}, we note that 
\begin{equation} \label{eq:limit_int_p}
 \lim\limits_{\lambda\to0} j(\lambda)>0 \hbox{ and } \lim\limits_{\lambda\to r} j(\lambda) = +\infty.
\end{equation}

Let us define
\begin{equation}\label{eq:intp,j}
    k(\lambda) := 2\sqrt{\frac{\alpha_\lambda^2 - 4\varepsilon^2}{r - \lambda}} + \frac{4\varepsilon}{\sqrt{\lambda}} - \lambda \geq j(\lambda).
\end{equation}
From \eqref{eq:alphalambda_expr}, we deduce:
\begin{equation} \label{eq:alphalambda}
    \alpha_\lambda = \sqrt{\frac{2\int_0^{2\varepsilon} f}{r - \lambda}},
\end{equation}
and therefore,
\begin{equation} \label{eq:condition_intf_r}
    k(r/2) <0 \Leftrightarrow  \int_0^{2\varepsilon} f < \frac{r^4}{128} + 2 \varepsilon^2 r -\frac{\sqrt{2} \varepsilon}{8} r^{5/2}.
\end{equation}
Thus, \eqref{eq:condition_intf_r} gives a sufficient condition for the existence of at least two zeroes $0<\lambda_1< r/2 < \lambda_2<r$ of the function $j$, corresponding to the existence of two stationary solutions $p_{\lambda_1}$ and $p_{\lambda_2}$ of~\eqref{eq:main}. Since  $\lambda\mapsto\alpha_\lambda$ is increasing in $[0,r)$ and using \eqref{eq:ode1}, the comparison principle for ODEs implies that  $p_{\lambda_1}<p_{\lambda_2}$.
\end{proof}

\section*{Acknowledgements}
This work was supported by the European Union's Horizon Europe research and innovation programme
through the BCOMING project (\emph{Biodiversity Conservation to Mitigate the Risks of Emerging Infectious Diseases}),
Grant Agreement No.~101059483.
This work was also supported by the ANR project ReaCh (ANR-23-CE40-0023-01). We also thank Olivier Bonnefon, Guillaume Fournié, Sylvain Gandon, and Samuel Soubeyrand for helpful discussions.


\bibliographystyle{siam}

\begin{thebibliography}{}

\end{thebibliography}


\begin{thebibliography}{10}

\bibitem{AlfCar17}
{\sc M.~Alfaro and R.~Carles}, {\em Replicator-mutator equations with quadratic fitness},
Proc. Amer. Math. Soc., 145 (2017), pp.~5315--5327.

\bibitem{Alf-Duc-Fay-20}
{\sc M.~Alfaro, A.~Ducrot, and G.~Faye}, {\em Quantitative estimates of the threshold phenomena for propagation in reaction-diffusion equations},
SIAM J. Appl. Dyn. Syst., 19 (2020), pp.~1291--1311.

\bibitem{AlfHamRoq24}
{\sc M.~Alfaro, F.~Hamel, and L.~Roques}, {\em Propagation or extinction in bistable equations: the non-monotone role of initial fragmentation},
Discrete Contin. Dyn. Syst. Ser. S, to appear (2024).

\bibitem{AlfVer18}
{\sc M.~Alfaro and M.~Veruete}, {\em Evolutionary branching via replicator-mutator equations},
J. Dyn. Differ. Equ., to appear (2018), pp.~1--24.

\bibitem{Bik-14}
{\sc V.~N. Biktashev}, {\em A simple mathematical model of gradual {D}arwinian evolution: emergence of a {G}aussian trait distribution in adaptation along a fitness gradient},
J. Math. Biol., 68 (2014), pp.~1225--1248.

\bibitem{DuMat10}
{\sc Y.~Du and H.~Matano}, {\em Convergence and sharp thresholds for propagation in nonlinear diffusion problems},
J. Eur. Math. Soc., 12 (2010), pp.~279--312.

\bibitem{Fle-79}
{\sc W.~H. Fleming}, {\em Equilibrium distributions of continuous polygenic traits},
SIAM J. Appl. Math., 36 (1979), pp.~148--168.

\bibitem{friedman-parabolic}
{\sc A.~Friedman}, {\em Partial Differential Equations of Parabolic Type},
Courier Dover Publications, Mineola, NY, 2008.

\bibitem{Gar-Roq-Ham-12}
{\sc J.~Garnier, L.~Roques, and F.~Hamel}, {\em Success rate of a biological invasion in terms of the spatial distribution of the founding population},
Bull. Math. Biol., 74 (2012), pp.~453--473.

\bibitem{HamLav20}
{\sc F.~Hamel, F.~Lavigne, G.~Martin, and L.~Roques}, {\em Dynamics of adaptation in an anisotropic phenotype-fitness landscape},
Nonlinear Anal. Real World Appl., 54 (2020), p.~103107.

\bibitem{Kim-65}
{\sc M.~Kimura}, {\em A stochastic model concerning the maintenance of genetic variability in quantitative characters},
Proc. Natl. Acad. Sci. USA, 54 (1965), pp.~731--736.

\bibitem{Lan-75}
{\sc R.~Lande}, {\em The maintenance of genetic variability by mutation in a polygenic character with linked loci},
Genet. Res., 26 (1975), pp.~221--235.

\bibitem{Lat-et-al-20}
{\sc A.~Latinne, B.~Hu, K.~J. Olival, G.~Zhu, L.~Zhang, H.~Li, A.~A. Chmura, H.~E. Field, C.~Zambrana-Torrelio, J.~H. Epstein, et~al.},
{\em Origin and cross-species transmission of bat coronaviruses in {C}hina},
Nat. Commun., 11 (2020), p.~4235.

\bibitem{Lau-et-al-05}
{\sc S.~K. Lau, P.~C. Woo, K.~S. Li, Y.~Huang, H.-W. Tsoi, B.~H. Wong, S.~S. Wong, S.-Y. Leung, K.-H. Chan, and K.-Y. Yuen},
{\em Severe acute respiratory syndrome coronavirus-like virus in {C}hinese horseshoe bats},
Proc. Natl. Acad. Sci. USA, 102 (2005), pp.~14040--14045.

\bibitem{MarRoq16}
{\sc G.~Martin and L.~Roques}, {\em The non-stationary dynamics of fitness distributions: asexual model with epistasis and standing variation},
Genetics, 204 (2016), pp.~1541--1558.

\bibitem{Mat-Pol-16}
{\sc H.~Matano and P.~Pol{\'a}cik}, {\em Dynamics of nonnegative solutions of one-dimensional reaction-diffusion equations with localized initial data. {P}art {I}: {A} general quasiconvergence theorem and its consequences},
Comm. Partial Differential Equations, 41 (2016), pp.~785--811.

\bibitem{Mol91}
{\sc D.~Mollison}, {\em Dependence of epidemic and population velocities on basic parameters},
Math. Biosci., 107 (1991), pp.~255--287.

\bibitem{Mur-Zho-13}
{\sc C.~B. Muratov and X.~Zhong}, {\em Threshold phenomena for symmetric decreasing solutions of reaction-diffusion equations},
NoDEA Nonlinear Differential Equations Appl., 20 (2013), pp.~1519--1552.

\bibitem{Mur-Zho-17}
{\sc C.~B. Muratov and X.~Zhong}, {\em Threshold phenomena for symmetric-decreasing radial solutions of reaction-diffusion equations},
Discrete Contin. Dyn. Syst., 37 (2017), pp.~915--944.

\bibitem{Nad-preprint-23}
{\sc G.~Nadin}, {\em On the instability of threshold solutions of reaction-diffusion equations, and applications to optimization problems},
arXiv:2311.07154 (2023).

\bibitem{Par-et-al-15}
{\sc C.~R. Parrish, P.~R. Murcia, and E.~C. Holmes}, {\em Influenza virus reservoirs and intermediate hosts: dogs, horses, and new possibilities for influenza virus exposure of humans},
J. Virol., 89 (2015), pp.~2990--2994.

\bibitem{Pol11}
{\sc P.~Pol{\'a}cik}, {\em Threshold solutions and sharp transitions for nonautonomous parabolic equations on $\mathbb{R}^n$},
Arch. Ration. Mech. Anal., 199 (2011), pp.~69--97.

\bibitem{readi2solve}
{\sc L.~Roques}, {\em Readi2solve: An Online Python Toolbox for Solving Reaction-Diffusion Problems},
Preprint, hal-04728969 (2024).

\bibitem{TsiLev96}
{\sc L.~S. Tsimring, H.~Levine, and D.~A. Kessler}, {\em {RNA} virus evolution via a fitness-space model},
Phys. Rev. Lett., 76 (1996), pp.~4440--4443.

\bibitem{Web-Web-01}
{\sc R.~Webby and R.~Webster}, {\em Emergence of influenza {A} viruses},
Phil. Trans. R. Soc. Lond. B, 356 (2001), pp.~1817--1828.

\bibitem{Zla06}
{\sc A.~Zlato\v{s}}, {\em Sharp transition between extinction and propagation of reaction},
J. Amer. Math. Soc., 19 (2006), pp.~251--263.

\end{thebibliography}

\end{document}